\documentclass[12pt,a4paper]{amsart}
\usepackage[colorlinks,citecolor=blue,urlcolor=blue,bookmarks=false,hypertexnames=true]{hyperref}
\usepackage[a4paper,margin=23mm,top=30mm]{geometry}

\usepackage{amsfonts, amsmath, amssymb, amsgen, amsthm, amscd}
\usepackage{newtxtext,newtxmath}
\usepackage[utf8]{inputenc} 

\usepackage{color}

\usepackage[all]{xy}

\usepackage{hyperref}
\usepackage{mathtools,slashed}
\usepackage{setspace}
\usepackage{enumerate}

\def\a{\alpha}
\def\b{\beta}
\def\d{\delta}
\def\D{\Delta}
\def\g{\gamma}

\def\s{\sigma}

\def\t{\theta}

\def\vp{\varphi}

\def\ot{\otimes}

\def\rt{\triangleright}
\def\lt{\triangleleft}

\def\D{\Delta}

\def\Ad{\mathop{\rm Ad}\nolimits}
\def\ad{\mathop{\rm ad}\nolimits}

\def\dt{\left.\frac{d}{dt}\right|_{_{t=0}}}
\def\ds{\left.\frac{d}{ds}\right|_{_{s=0}}}

\usepackage{enumerate}

\newcommand{\G}[1]{\mathfrak{#1}}
\newcommand{\C}[1]{\mathcal{#1}}

\newcommand{\btr}{\blacktriangleright}

\renewcommand{\leq}{\leqslant}
\renewcommand{\geq}{\geqslant}

\numberwithin{equation}{section}

\newtheorem{theorem}{Theorem}[section]
\newtheorem{proposition}{Proposition}[section]

\newtheorem{corollary}[theorem]{Corollary}
\theoremstyle{definition}

\newtheorem*{remark}{Remark}
\newtheorem{example}[theorem]{Example}

\usepackage{geometry}
\title{Dynamics over Homogeneous Spaces}

\author{FİLİZ ÇAĞATAY UÇGUN}
\address{Department of Mathematics, Maltepe University, 34857 Maltepe-Ä°stanbul, Turkey}
\email{filizcagatayucgun@maltepe.edu.tr}

\author{O\u{g}ul Esen}
\address{Department of Mathematics, Gebze Technical University,  41400 Gebze-Kocaeli, Turkey}
\email{oesen@gtu.edu.tr}

\author{Serkan Sütlü}
\address{Department of Mathematics, Gebze Technical University,  41400 Gebze-Kocaeli, Turkey}

\email{serkansutlu@gtu.edu.tr}

\begin{document}

\begin{abstract}
We present the Euler-Lagrange and Hamilton's equations for a system whose configuration space is a unified product Lie group $G=M\bowtie_\g H$, for some $\g:M\times M \to H$. By reduction, then, we obtain the Euler-Lagrange type and Hamilton's type equations of the same form for the quotient space $M\cong G/H$, although it is not necessarily a Lie group. We observe, through further reduction, that it is possible to formulate the Euler-Poincar\'{e} type and Lie-Poisson type equations on the corresponding quotient $\G{m}\cong \G{g}/\G{h}$ of Lie algebras, which is not a priori a Lie algebra. Moreover, we realize the $n$th order iterated tangent group $T^{(n)}G$ of a Lie group $G$ as an extension of the $n$th order tangent group $T^nG$ of the same type. More precisely, $\G{g}$ being the Lie algebra of $G$, $T^{(n)}G \cong  \G{g}^{\times \,2^n-1-n} \bowtie_\g T^nG$ for some $\g:\G{g}^{\times \,2^n-1-n} \times \G{g}^{\times \,2^n-1-n} \to T^nG$. We thus obtain the $n$th order Euler-Lagrange (and then the $n$th order Euler-Poincar\'e) equations over $T^nG$ by reduction from those on $T(T^{n-1}G)$. Finally, we illustrate our results in the realm of the Kepler problem, and the non-linear tokamak plasma dynamics.
\end{abstract}

\keywords{Unified product; Lie-Poisson dynamics; Euler-Poincar\'{e} dynamics}

\maketitle

\tableofcontents

\setlength{\parskip}{0.75cm}
\onehalfspace

\section{Introduction}

A \emph{Klein geometry} is defined as a pair $(G,H)$ of Lie groups, where $H \subseteq G$ is a closed subgroup. The group $G$ is then called the \emph{principal group}, while the quotient $G/H$, which happens to be a smooth manifold, is referred as the \emph{space} of the Klein geometry.

The classical geometries may all be realized as Klein geometries. For instance, the Euclidean geometry is encoded by the pair $(E(n),O(n))$ of the Euclidean group $E(n)\cong \mathbb{R}^n\rtimes O(n)$, and the orthogonal group $O(n)$. The underlying space, then, is the Euclidean space $\mathbb{R}^n$. Similarly, the affine geometry is modeled by the pair $(Aff(n),GL(n))$ of the affine group $Aff(n) \cong \mathbb{R}^n\rtimes GL(n)$ and the general linear group $GL(n)$, in which case the underlying space is the affine space. The spherical geometry, on the other hand, is represented by $(O(n+1),O(n))$, and its underlying space happens to be the $n$-sphere $S^n$. We refer the reader to \cite{Sharpe-book} for further details, and the Klein models of the projective, hyperbolic, conformal, and elliptic geometries.

A Klein geometry $(G,H)$ is called \emph{reductive}, in which case the underlying space $G/H$ is named as a \emph{reductive homogeneous space}, if the Lie algebra $\G{h}$ of $H$ has an $H$-invariant complement in the Lie algebra $\G{g}$ of $G$. Any \emph{symmetric space}\footnote{A symmetric space is defined to be a Riemannian manifold with an invariant (under parallel translations) curvature tensor, \cite{Helgason-book}.}, for instance, gives rise to a reductive homogeneous space, \cite{Helgason-book}. In this case, the Lie algebra $\G{g}$ decomposes as
\[
\G{g} = \G{m} \oplus \G{h},
\]
so that
\begin{equation}\label{bracket-pattern}
[\G{h}, \G{h}] \subseteq \G{h}, \qquad [\G{h}, \G{m}] \subseteq \G{m}, \qquad [\G{m}, \G{m}] \subseteq \G{h}.
\end{equation}
A comparable bracket pattern may be observed in the structure of semisimple Lie algebras, \cite[Subsect. 6.1.1]{KlimSchm-book}. Let $\G{g}$ be a (finite-dimensional) complex semisimple Lie algebra. Then, $\G{h} \subseteq \G{g}$ being the Cartan subalgebra, we have
\[
\G{g} = \left( \bigoplus_{\a \neq 0}\,\G{g}_\a \right) \oplus \G{h},
\]
where each $\a$ is a (non-wanishing) linear form on $\G{h}$, called a \emph{root}, and $\G{g}_\a$ is the eigenspace of the operator $\a$, called the corresponding \emph{root subspace}. Moreover, setting a basis $H_1,H_2, \ldots, H_\ell$ of $\G{h}$, it is possible to choose basis vector $E_\a$ in each $\G{g}_\a$ so that $[E_\a,E_{-\a}] = H_\a \in \G{h}$. Then, $\D$ denoting the set of all roots $\a$, the set
\[
H_1,H_2, \ldots, H_\ell, \quad E_\a, \,\, \a \in \D
\] 
happens to be a basis for $\G{g}$, called the \emph{Cartan-Weyl basis}. The bracket operation on $\G{g}$ may then be given by
\[
[H_i,H_j] = 0, \quad [H_i,E_\a] = \a(H_i)E_\a, \quad [E_\a,E_{-\a}] = H_\a, \quad [E_\a,E_\b] =N_{\a\b}E_{\a+\b}
\]
for any $\a, \b\in \D$, so that $\a +\b \neq 0$, and any $1 \leq i,\,j\leq \ell$, where $N_{\a\b} = 0$ whenever $\a+\b \notin \D$.


The algebraic foundations of such Lie algebra decompositions has first been studied thoroughly in \cite{AgorMili14}, wherein $\G{g} = \G{m}\oplus \G{h}$ is called the \emph{unified product} of $\G{h}$ and $\G{m}$.

On the other hand, for a physical system whose configuration space is a manifold, the Lagrangian realization of the dynamics is determined by a Lagrangian function on the tangent bundle of the configuration space by means of a variational principle, while the Hamiltonian realization is encoded by a Hamiltonian function over the cotangent bundle of the configuration space, \cite{abraham1978foundations,arnold1989mathematical}. 

In case the configuration space is a Lie group $G$, then relying on the (semi-direct product) trivialization $TG \cong G \ltimes \G{g}$ of the tangent group (into $G$ and its Lie algebra $\G{g}$), the equations (of motion) that govern the dynamics are given by
\begin{equation}\label{EL-eqn}
\frac{d}{dt}\frac{\d \C{L}}{\d \xi} = T^\ast L_g\left(\frac{\d \C{L}}{\d g}\right) - \ad^\ast_\xi\left(\frac{\d \C{L}}{\d \xi}\right),
\end{equation}
which are known as the (trivialized) \emph{Euler-Lagrange equations} of the (trivialized) Lagrangian $\C{L}:G\ltimes \G{g} \to \mathbb{R}$, $\C{L}=\C{L}(g,\xi)$, see for instance \cite{EsenSutl17}.

From the Hamiltonian point of view, on the other hand, appealing to the trivialization $T^\ast G \cong \G{g}^\ast \rtimes G$ of the cotangent group, the equations (of motion) that represent the dynamics of the system are given by
\begin{align}\label{HamiltonEqns}
\begin{split}
& \frac{d\mu}{dt} = -T^\ast R_g\left(\frac{\d \C{H}}{\d g}\right) + \ad^\ast_{\frac{\d \C{H}}{\d \mu}}(\mu), \\
& \frac{dg}{dt} = TR_g\left(\frac{\d \C{H}}{\d \mu}\right),
\end{split}
\end{align}
which are called the \emph{Hamilton's equations} of the (trivialized) Hamiltonian $\C{H}:\G{g}^\ast \rtimes G \to \mathbb{R}$, $\C{H}=\C{H}(\mu,g)$, \cite{EsenSutl16}.

The theory of reduction in mechanical systems involves investigating how the dimensionality of the system's phase space can be reduced through the utilization of available symmetries and associated conservation laws. In the case the systems configuration space is a Lie group $G$, both the tangent group $TG$ and the cotangent group $T^\ast G$ serve as excellent examples of systems that exhibit symmetry. In either of these cases the symmetry is encoded by the action of the Lie group $G$.

In the Lagrangian framework, the reduction of \eqref{EL-eqn} by the $G$-action then yields the \emph{Euler-Poincar\'e equations}
\begin{equation}\label{EP-eqn}
\frac{d}{dt}\frac{\d \ell}{\d \xi} = - \ad^\ast_\xi\left(\frac{\d \ell}{\d \xi}\right),
\end{equation}
of the reduced Lagrangian $\ell:G \backslash TG \cong \G{g} \to \mathbb{R}$, $\ell=\ell(\xi)$. For further details of the Lagrangian reduction we refer the reader to \cite{cendra1998lagrangian,n2001lagrangian}.  The Hamiltonian counterpart of the reduction, on the other hand, produces the \emph{Lie-Poisson equations}
\begin{equation}\label{LP-Eqns}
\frac{d\mu}{dt} =  \ad^\ast_{\frac{\d \G{H}}{\d \mu}}(\mu), 
\end{equation}
of the reduced Hamiltonian $\G{H}:T^\ast G / G \cong \G{g}^\ast  \to \mathbb{R}$, $\G{H}=\G{H}(\mu)$. For more about the Hamiltonian reduction the reader may consult to \cite{MaWe74,MaMiOrPeRa07,MaRa86}.

In the present manuscript we delve into the Euler-Poincar\'e equations on a unified product Lie algebra $\G{g} = \G{m}\oplus \G{h}$, and the corresponding Lie-Poisson equations defined on the dual space $\G{g}^\ast$. Within both the Lagrangian and Hamiltonian formalisms, reduction by $\G{h}$ yields equations that represent the dynamics over the quotient space $\G{m}$ and its dual $\G{m}^\ast$ respectively. These equations, as such, are the most natural substitutes of the Euler-Poincar\'e equations (resp. the Lie-Poisson equations) over the quotient space $\G{m}\cong\G{g}/\G{h}$ (resp. $\G{m}^\ast$), which is not initially a Lie algebra. 

The Lie group of a unified product Lie algebra $\G{g} = \G{m} \oplus \G{h}$ is a \emph{unified product} $G$ in the level of Lie groups, \cite{AgorMili14-II}. More explicitly, the group $G$ is entirely determined by a subgroup $H \subseteq G$, along with a complement $M\subseteq G$ (in the sense that $G=M\times H$ as manifolds) which is not necessarily a subgroup. Among other things, we shall hereby show that the tangent (resp. cotangent) group of a unified product is also a unified product Lie group. Consequently, the Euler-Lagrange equations on the tangent group $TG \cong TM\times TH$ will lead, by the $TH$ reduction, to the Euler-Lagrange type equations on $TM = T(G/H)$ even though $M = G/H$ is not a priori a Lie group. The recipe functions identically in the Hamiltonian counterpart of the dynamics. Specifically, we shall obtain the Hamilton's equations on the cotangent group $T^\ast G \cong T^\ast M\times T^\ast H$. Then, by the $T^\ast H$ reduction we shall derive the Hamilton's type equations over $T^\ast M = T^\ast(G/H)$, once again, despite $M = G/H$ is not inherently a Lie group.

We shall illustrate our results on two examples. We first realize the Lie algebra structure on the energy-momentum space of the Kepler problem considered in \cite{Marl12} as a unified product. Accordingly we shall obtain the Euler-Poincar\'e equations over the energy-momentum space, as well as the Lie-Poisson equations over its dual. As a second illustration we will consider the Hamiltonian four-field model, proposed in \cite{HazeKotsMorr85}, see also \cite{HazeHsuMorr87}, for two-dimensional compressible reduced magnetohydrodynamics. The model offers a simplified representation of nonlinear tokamak dynamics, accommodating finite ion Larmor radius physics, along with other effects associated with compressibility and electron adiabaticity. The Lie algebra that underlies the four field model was worked out in \cite{ThifMorr00} either, from the point of view of the Lie algebra extensions. We shall represent this Lie algebra as a unified product, and will thus obtain the equations of motion both in Lagrangian and Hamiltonian frameworks. 

The flexible decomposition offered by a unified product allows to study the higher order dynamics as well. More precisely, we observe that the $n$th iterated tangent group 
\[
T^{(n)}G := \underset{n-many}{\underbrace{T(T(T\ldots (T}}G)))
\]
of a Lie group $G$ is a unified product of the $n$th order tangent group $T^nG$, and $2^n-(n+1)$-many copies of the Lie algebra $\G{g}$ of $G$;
\[
T^{(n)}G \cong  \G{g}^{\times \,2^n-1-n} \times T^nG.
\]
As such, we shall be able to access the dynamics (both Lagrangian and Hamiltonian) on $T^nG$ from that on $T^{(n)}G$ through $\G{g}^{\times \,2^n-1-n}$-reduction. We shall illustrate all details in the case $n=3$.

\subsubsection*{Notation and Conventions}
 
We shall consider decompositions of Lie algebras, and Lie groups. As for Lie algebras we shall consider the decomposition $\G{g} \cong \G{m}\,\oplus\,\G{h}$ of a Lie algebra into a vector subspace $\G{m} \subseteq \G{g}$ and a Lie subalgebra $\G{h} \subseteq \G{g}$. Their generic elements are denoted by $v,v_1,v_2,...\in \G{m}$, $\eta,\eta_1,\eta_2,... \in \G{h}$, with $(v_1,\eta_1),(v_2,\eta_2) \in \G{g}$. Similarly, we shall consider the decomposition $G\cong M \times H$ of a Lie group into a submanifold $M\subseteq G$ and a subgroup $H\subseteq G$. We shall, in this case, refer the generic elements as $m,m_1,m_2,...\in M$ and $h,h_1,h_2,... \in H$, as well as $(m_1,h_1),(m_2,h_2) \in G$.

\section{Cocycle Double Cross Constructions}

\subsection{Cocycle Double Cross Product Lie Groups}\label{Sec-cdcplg}~

We shall first recall from \cite[Subsec. 2.2]{EsenGuhaSutl22} the \emph{cocycle double cross product} Lie group construction. This construction has first been introduced in \cite[Sec. 3]{AgorMili14-II}, and was named as the \emph{unified product} construction. In an attempt to emphasize the presence of the cocycle term, and clarify its place in a hierarchy of extensions, the term \emph{cocycle double cross product} was adopted in \cite{EsenGuhaSutl22}. In the present text we shall stick to the notation and the terminology of \cite{EsenGuhaSutl22}.

Let $M$ be a manifold with a distinguished element $o\in M$, and let $H$ be a Lie group acting on $M$ via 
\begin{equation}\label{left-action-gr}
\vp:H\times M \to M, \qquad (h, m)\mapsto h \rt m,
\end{equation}
so that, $1\in H$ being the identity element,
\[
1 \rt m = m, \qquad h\rt o = o
\]
are satisfied. Let, further, the pair $(M,H)$ be equipped with the (smooth) maps 
\begin{align}
& \s:H\times M \to H, \qquad (h, m)\mapsto \s(h, m),\label{right-action-gr} \\
& \Phi:M\times M \to M, \qquad (m_1, m_2)\mapsto \Phi(m_1,m_2)=:m_1 \cdot m_2, \label{phi-map-gr} \\
& \g:M\times M \to H, \qquad (m_1, m_2)\mapsto \g(m_1, m_2) , \label{theta-map-gr} 
\end{align}
satisfying
\begin{equation}\label{normalization-gr}
\s(h,o) = h, \qquad \s(1,m) = 1, \qquad \g(o,o) = 1.
\end{equation}
Along the lines of \cite{EsenGuhaSutl22}, we shall call the pair $(M,H)$ a \emph{cocycle matched pair} if 
\begin{align}
& o\cdot m = m = m\cdot o, \label{phi-map-gr-e} \\
& \g(m,o) = 1 = \g(o,m), \label{theta-map-gr-e} \\
& h\rt (m_1\cdot m_2) = (h\rt m_1)\cdot (\s(h, m_1) \rt m_2), \label{left-action-gr-on-multp} \\
& \s(h,m_1\cdot m_2)\ast\g(m_1, m_2) = \g(h\rt m_1, \s(h, m_1) \rt m_2)\ast\s(\s(h,m_1), m_2), \label{right-action-theta-comp-gr} \\
& \s(h_1\ast h_2, m) = \s(h_1, h_2\rt m)\ast\s(h_2, m),  \label{right-action-gr-on-multp}  \\
& m_1\cdot (m_2\cdot m_3) = (m_1\cdot m_2) \cdot (\g(m_1, m_2)\rt m_3), \label{phi-map-gr-assoc} \\
& \g(m_1, m_2\cdot m_3)\ast\g(m_2, m_3) = \g(m_1\cdot m_2, \g(m_1, m_2)\rt m_3)\ast\s(\g(m_1, m_2), m_3),  \label{theta-map-gr-cocycle} \\
& \text{for any } m\in M \text{ there is } m^r\in M \text{ such that } m\cdot  m^r = o,\label{inverses-in-M-gr}
\end{align}
are satisfied, where $H\times H\to H$, $(h_1,h_2)\mapsto h_1\ast h_2$ denotes the multiplication in $H$. In this case, as was showed in  \cite[Thm. 3.5]{AgorMili14-II} and \cite[Prop. 2.3]{EsenGuhaSutl22}, the product space $M\bowtie_\g H:= M\times H$ happens to be a Lie group by
\begin{equation}\label{cocycle-double-cross-prod-multp}
(m_1,h_1)(m_2,h_2) = \Big(m_1\cdot (h_1\rt m_2 ),\,\g(m_1,h_1\rt m_2)\ast\s(h_1 ,m_2)\ast h_2 \Big),
\end{equation}
with $(o,1) \in M\times H$ being the unit. We shall call this Lie group the \emph{cocycle double cross product} of the pair $(M,H)$. Let us note also that, given any $(m,h) \in M\bowtie_\g H$, we may observe at once that
\begin{equation}\label{inverse-cdcp}
 (m,h)^{-1} =  \Big(h^{-1}\rt m^r,\,\s(h^{-1},m^r)\ast \g(m,m^r)^{-1}\Big).
\end{equation}

In comparison with \cite[Thm. 7.2.3]{Majid-book}, \cite[Thm. 3.1]{AgorMili14-II}, and \cite[Prop. 2.3]{BespDrab01}, the cocycle double cross product construction has the following property, see also \cite[Prop. 2.4]{EsenGuhaSutl22}. 

\begin{proposition}\label{prop-universal-gr}
Let $M$ be a (smooth) manifold, and let $G,H$ be two Lie groups so that
\[
\xymatrix{
M \ar@{^{(}->}[r]_i  & G    &  \ar@{_{(}->}[l]^j_{\rm gr}H,
}
\]
as manifolds. Let also $\mu:G\times G\to G$ being for the multiplication in $G$, the map $\mu\circ (i\times j):M\times H \to G$ be a diffeomorphism. Then, $G$ is a double cross product of the pair $(M,H)$. In this case, the maps \eqref{left-action-gr}, \eqref{right-action-gr}, \eqref{phi-map-gr}, and \eqref{theta-map-gr} are obtained from
\[
hm = (h\rt m)\s(h,m), \qquad m_1m_2 = (m_1\cdot m_2)\g(m_1,m_2),
\]
for any $m_1,m_2\in M$, and any $h\in H$.
\end{proposition}

\begin{remark}
It worths to remark that the cocycle double cross product (or unified product) construction provides a unique frame under which the double cross products of \cite{Maji90} and 2-cocycle extensions may be collected. More precisely, given a cocycle double cross product $G = M\bowtie_\g H$, if the twisted cocycle $\gamma: M\times M \to H$ is trivial, then $M\subseteq G$ becomes a subgroup, and $G$ becomes the double cross product $M\bowtie H$ of $M$ and $H$ in the sense of \cite{Maji90}, see also \cite{Majid-book}. If, on the other extreme, $H$ becomes abelian, along with a trivial action on $M$, then $M$ happens to be a Lie group, $\s:H\times M \to H$ becomes a group action, and $\g:M\times M\to H$ turns out to be a usual 2-cocycle in the (Lie group) cohomology of $M$ with coefficients in $H$. In this case, $G$ is said to be the 2-cocycle extension of $M$ by $H$, which is denoted by $M \ltimes_\g H$.
\end{remark}

\subsection{Cocycle double cross sum Lie algebras}\label{subsect-cocycle-double-cross-sum-Lie-alg}~

Let us next recall the tangent spaces, at identity, of cocycle double cross product Lie groups. These tangent spaces have been introduced in \cite{AgorMili14}, and they were called the \emph{unified product Lie algebras}, see also \cite{AgoreMilitaru-book}. We shall, however, prefer the phrase \emph{cocycle double cross sum Lie algebra} along the lines of \cite{EsenGuhaSutl22} as we did for the Lie groups.

Let $\G{g}$ be a Lie algebra, and  $\G{h} \subseteq \G{g}$ be a Lie subalgebra. Let also  $\G{m}$ be the vector space complement of $\G{h}$ in $\G{g}$, which is not assumed a priori a Lie subalgebra. Let, further, $\G{m}$ has the structure of a left $\G{h}$-module through
\begin{equation}\label{left-action}
\rt: \G{h}\ot \G{m}\to \G{m}, \qquad \eta\ot v\mapsto \eta\rt v,
\end{equation}
and the pair $(\G{m},\G{h})$ is equipped with linear maps
\begin{align}
&\phi :\G{m}\ot \G{m} \to \G{m}, \qquad v_1\ot v_2\mapsto \phi(v_1\ot v_2) =:\phi(v_1,v_2) =: [[ v_1,v_2]] \label{phi-map} \\
&\t :\G{m}\ot \G{m} \to \G{h}, \qquad v_1\ot v_2\mapsto \t(v_1\ot v_2) =:\t(v_1,v_2),  \label{theta-map} \\
&\psi: \G{h}\ot \G{m}\to \G{h}, \qquad \psi(\eta\ot v) =:\psi(\eta,v). \label{right-action}
\end{align}
Then it was showed in \cite[Thm. 3.2]{AgorMili14}, see also \cite[Prop. 2.1]{EsenGuhaSutl22}, that $\G{m}\bowtie_\t\G{h}:=\G{m}\oplus \G{h}$ is a Lie algebra through
\begin{equation}\label{cocycle-double-cross-sum-bracket}
[(v_1, \eta_1), (v_2, \eta_2)] = \Big([[ v_1,v_2]] + \eta_1\rt v_2 - \eta_2\rt v_1,  [\eta_1,\eta_2] + \psi(\eta_1 ,v_2) - \psi(\eta_2, v_1) +\t(v_1,v_2) \Big),
\end{equation}
if and only if 
\begin{align}
& [[ v,v ]] = 0, \qquad \qquad \t(v,v) = 0, \label{phi-xi-xi-theta-xi-xi}\\
& \eta\rt [[v_1,v_2]] = [[\eta\rt v_1, v_2]] + [[v_1,\eta\rt v_2]] + \psi(\eta,v_1)\rt v_2 - \psi(\eta, v_2)\rt v_1, \label{h-on-phi} \\
& [\eta,\t(v_1,v_2)]  = \t(\eta\rt v_1, v_2) + \t(v_1, \eta\rt v_2) + \psi(\psi(\eta,v_1), v_2) - \psi(\psi(\eta,v_2)), v_1)  - \psi(\eta, [[v_1,v_2]]), \label{h-on-theta-and-psi-being-action} \\
&\psi([\eta_1,\eta_2] ,v)=  [\eta_1,\psi(\eta_2,v)]+[\psi(\eta_1,v),\eta_2] + \psi(\eta_1,\eta_2\rt v) - \psi(\eta_2,\eta_1\rt v), \label{h-bracket-psi-comp} \\
& \circlearrowright  [[[[v_1,v_2]],v_3]] + \circlearrowright \t(v_1,v_2)\rt v_3 = 0, \label{phi-Jacobi-and-theta-action} \\
& \circlearrowright \psi(\t(v_1,v_2),v_3) + \circlearrowright \t([[v_1,v_2]],v_3) = 0, \label{cocycle-condition-for-theta}
\end{align}
are satisfied for any $v,v_1,v_2,v_3\in \G{m}$ and any $\eta,\eta_1,\eta_2 \in \G{h}$, where $\circlearrowright$ denotes summation over the cyclic permutations of $(v_1,v_2,v_3)$.

Following \cite{EsenGuhaSutl22}, we shall call the Lie algebra $\G{m}\bowtie_\t\G{h}:= \G{m}\oplus\G{h}$ the \emph{cocycle double cross sum} of $\G{m}$ and $\G{h}$. 

As was pointed out in \cite[Thm. 3.4]{AgorMili14}, and also \cite[Thm. 3.1.4]{AgoreMilitaru-book}, the Lie algebra $\G{m}\bowtie_\t\G{h}:= \G{m}\oplus\G{h}$ is completely determined by $\G{h}$ and its (vector space) complement $\G{m}$. More precisely, we have the following.

\begin{proposition}\label{prop-cocycle-double-sum-universal}
Let $\G{g}$ be a Lie algebra, $\G{h} \subseteq \G{g}$ a subalgebra, and let $\G{m}$ be the vector space complement of $\G{h}$ in $\G{g}$, that is, $\G{g}\cong \G{m}\oplus \G{h}$ as vector spaces. Then $\G{g}$ is a double cross sum of the pair $(\G{m}, \G{h})$. In this case, the linear maps \eqref{phi-map}, \eqref{theta-map}, \eqref{left-action}, and \eqref{right-action} may be obtained through
\[
[\eta,v] = \eta\rt v + \psi(\eta , v), \qquad [v_1,v_2] = [[ v_1,v_2]] + \t(v_1,v_2),
\]
for any $v,v_1,v_2\in \G{m}$, and any $\eta \in \G{h}$.
\end{proposition}

\begin{remark}
In accordance with the Lie groups, both double cross sum Lie algebras of \cite{Maji90,Majid-book} and 2-cocycle extensions are particular instances of cocycle double cross sums (unified products). Explicitly, given $\G{g} = \G{m}\bowtie_\t \G{h}$, if the twisted cocycle $\t:\G{m}\otimes \G{m} \to \G{h}$ is trivial, then $\G{g} = \G{m}\bowtie \G{h}$, that is a double cross sum Lie algebra. If, on the other hand, $\G{h}$ is abelian and the $\G{h}$-action on $\G{m}$ is trivial, then $\G{m} \subseteq \G{g}$ becomes a Lie algebra (though not a Lie subalgebra) and $\t:\G{m}\otimes \G{m} \to \G{h}$ becomes a 2-cocycle in the Lie algebra cohomology of $\G{m}$, with coefficients in $\G{h}$. In this case the notation $\G{g} = \G{m}\ltimes_\t \G{h}$ is adopted.
\end{remark}

\section{Lagrangian Dynamics on Cocycle Double Cross Product Lie groups}

\subsection{The tangent group of a cocycle double cross product}~

We shall now present the cocycle double cross product structure of the tangent group $TG$ of a cocycle double cross product Lie group $G:=M\bowtie_\g H$, in view of Proposition \ref{prop-universal-gr}. To this end, we shall make use of the (left) trivialization $tr_{TG}^L:TG\to G\ltimes \G{g}$, where $\G{g}$ stands for the Lie algebra of $G$. Accordingly, let us begin with the following observation on the Lie algebra of a cocycle double cross product group.

\begin{remark}
Given a cocycle double cross product group $G:=M\bowtie_\g H$, let $\G{h}$ be the Lie algebra of $H$, and $1 \in H$ denotes the identity element. Let also $o$ be the distinguished element of $M$. It, then, follows at once that the tangent space $T_{(o,1)}G = \G{g}$ may be identified with $T_oM \oplus \G{h}$, as vector spaces. On the other hand, the vector space $T_oM \oplus \G{h}$ has the structure of a Lie algebra, which admits $\G{h}$ as a subalgebra. As such, it follows from Proposition \ref{prop-cocycle-double-sum-universal} that $T_oM \oplus \G{h}$ is a double cross sum Lie algebra for some (twisted cocycle) $\t:T_oM \ot T_oM \to \G{h}$. We shall therefore present the Lie algebra of $M\bowtie_\g H$ as $T_oM \bowtie_{\t} \G{h}$.
\end{remark}

\begin{proposition}\label{prop-Ad-cocycle-cross}
The adjoint action of a cocycle double cross product $M\bowtie_\g H$ on its Lie algebra $T_oM \bowtie_{\t} \G{h}$ is given by
\begin{align}\label{Ad-cocycle-double-cross-prod}
\begin{split}
&\Ad_{(m,h)^{-1}}(v,\eta) = \Big(h^{-1}\rt T\G{L}_{m^r} \big[\g(m,m^r)^{-1} \rt T\G{R}_m v+TL_{\g(m,m^r)^{-1}}(\eta)\rt m)\big], \\
& TR_{\g(m^r, \g(m,m^r)^{-1}\rt  m) \ast \s(\g(m,m^r)^{-1}, m) \ast h}\big(\s(h^{-1}, m^r\cdot[\g(m,m^r)^{-1}\rt T\G{R}_m v])+\\
& \hspace{8cm}\s(h^{-1}, m^r\cdot[TL_{\g(m,m^r)^{-1}}(\eta) \rt m])\big) + \\
&TL_{\s(h^{-1}, m^r\cdot[\g(m,m^r)^{-1}\rt m])}TR_{\s(\g(m,m^r)^{-1}, m) \ast h} \big[\g(m^r, \g(m,m^r)^{-1}\rt T\G{R}_m v) +\\
& \hspace{9cm} \g(m^r, TL_{\g(m,m^r)^{-1}}(\eta)\rt m) + \\
& TL_{\g(m^r, \g(m,m^r)^{-1}\rt m)} \big(\g(\g(m,m^r)^{-1}\rt v, \g(m,m^r)^{-1}\rt m)\big)\big] +\\
& TL_{\s(h^{-1}, m^r\cdot[\g(m,m^r)^{-1}\rt m])\ast \g(m^r, \g(m,m^r)^{-1}\rt m)}TR_h\big[\s(\s(\g(m,m^r)^{-1},v), m)+\s(TL_{\g(m,m^r)^{-1}}(\eta), m)\big]\Big)
\end{split}
\end{align}
for any $(m,h)\in M\bowtie_\g H$, and any $(v,\eta) \in T_oM \bowtie_{\t} \G{h}$, where for any $m\in M$
\begin{align*}
& \G{L}_m:M \to M, \qquad n\mapsto m\cdot n,\\
& \G{R}_m:M \to M, \qquad n\mapsto n\cdot m.
\end{align*}
\end{proposition}

\begin{proof}
In view of \eqref{inverse-cdcp}, given a path $(v_t, \eta_t) \in M\bowtie_\g H$ with $v_0 = o \in M, \, v'_0 = v \in T_oM$, and $\eta_0 = 1 \in H,\, \eta'_0=\eta\in \G{h}$, we have
\begin{align}\label{Ad-group}
\begin{split}
&\Ad_{(m,h)^{-1}}(v_t,\eta_t) = (h^{-1}\rt m^r, \s(h^{-1}, m^r)\ast \g(m,m^r)^{-1})(v_t,\eta_t)(m,h) = \\
& \Big(h^{-1}\rt m^r, \s(h^{-1}, m^r)\ast \g(m,m^r)^{-1}\Big)\Big(v_t \cdot (\eta_t\rt m),\g(v_t, \eta_t \rt m)\ast \s(\eta_t, m)\ast h\Big) = \\
& \Big((h^{-1}\rt m^r)\cdot (A\rt [v_t \cdot (\eta_t\rt m)]), \\
& \hspace{2cm} \g(h^{-1}\rt m^r, A\rt [v_t \cdot (\eta_t\rt m)])\ast \s(A, v_t \cdot (\eta_t\rt m)) \ast \g(v_t, \eta_t \rt m)\ast \s(\eta_t, m)\ast h\Big)
\end{split}
\end{align}
where $A:=\s(h^{-1}, m^r)\ast \g(m,m^r)^{-1} \in H$. As for the derivative at $t=0$, we first note that
\begin{align*}
& \dt\, (h^{-1}\rt m^r)\cdot \Big(A\rt [v_t \cdot (\eta_t\rt m)]\Big)  =  \dt\, h^{-1}\rt \Big(m^r \cdot [\g(m,m^r)^{-1} \rt (v_t\cdot (\eta_t\rt m))]\Big) = \\
& \dt\, h^{-1}\rt \Big(m^r \cdot [\g(m,m^r)^{-1} \rt (v_t\cdot m)]\Big) + \dt\, h^{-1}\rt \Big(m^r \cdot [\g(m,m^r)^{-1} \ast\eta_t\rt m]\Big) = \\
& h^{-1}\rt T\G{L}_{m^r} \Big(\g(m,m^r)^{-1} \rt (T\G{R}_m v+\eta\rt m)\Big) = h^{-1}\rt T\G{L}_{m^r} \Big(\g(m,m^r)^{-1} \rt T\G{R}_m v+TL_{\g(m,m^r)^{-1}}(\eta)\rt m)\Big),
\end{align*}
where on the first equality we used \eqref{left-action-gr-on-multp}. Next,
\begin{align*}
& \dt\,\g(h^{-1}\rt m^r, A\rt [v_t \cdot (\eta_t\rt m)])\ast \s(A, v_t \cdot (\eta_t\rt m)) \ast \g(v_t, \eta_t \rt m)\ast \s(\eta_t, m)\ast h = \\
&  \dt\, \s(h^{-1}, m^r\cdot[\g(m,m^r)^{-1}\rt (v_t\cdot(\eta_t\rt m))])\ast \g(m^r, \g(m,m^r)^{-1}\rt (v_t\cdot(\eta_t\rt m))) \ast \\
&\hspace{4cm} \s(\g(m,m^r)^{-1}, v_t\cdot(\eta_t\rt m)) \ast \g(v_t, \eta_t \rt m)\ast \s(\eta_t, m)\ast h = \\
&  \dt\, \s(h^{-1}, m^r\cdot[\g(m,m^r)^{-1}\rt (v_t\cdot(\eta_t\rt m))])\ast \g(m^r, \g(m,m^r)^{-1}\rt (v_t\cdot(\eta_t\rt m))) \ast \\
&\hspace{2cm} \g(\g(m,m^r)^{-1}\rt v_t, \s(\g(m,m^r)^{-1},v_t)\ast\eta_t\rt m) \ast \s(\s(\g(m,m^r)^{-1},v_t)\ast\eta_t, m)\ast h = \\
& \s(h^{-1}, m^r\cdot[\g(m,m^r)^{-1}\rt (v\cdot m)])\ast \g(m^r, \g(m,m^r)^{-1}\rt  m) \ast \s(\g(m,m^r)^{-1}, m) \ast h + \\
& \s(h^{-1}, m^r\cdot[\g(m,m^r)^{-1}\ast\eta \rt m])\ast \g(m^r, \g(m,m^r)^{-1}\rt m) \ast \s(\g(m,m^r)^{-1}, m) \ast h + \\
& \s(h^{-1}, m^r\cdot[\g(m,m^r)^{-1}\rt m])\ast \g(m^r, \g(m,m^r)^{-1}\rt (v\cdot m)) \ast \s(\g(m,m^r)^{-1}, m) \ast h +\\
& \s(h^{-1}, m^r\cdot[\g(m,m^r)^{-1}\rt m])\ast \g(m^r, \g(m,m^r)^{-1}\ast\eta\rt m) \ast \s(\g(m,m^r)^{-1}, m) \ast h + \\
& \s(h^{-1}, m^r\cdot[\g(m,m^r)^{-1}\rt m])\ast \g(m^r, \g(m,m^r)^{-1}\rt m) \ast \\
&\hspace{5cm} \g(\g(m,m^r)^{-1}\rt v, \g(m,m^r)^{-1}\rt m) \ast \s(\g(m,m^r)^{-1}, m) \ast h +\\
&\s(h^{-1}, m^r\cdot[\g(m,m^r)^{-1}\rt m])\ast \g(m^r, \g(m,m^r)^{-1}\rt m) \ast \s(\s(\g(m,m^r)^{-1},v), m) \ast h + \\
& \s(h^{-1}, m^r\cdot[\g(m,m^r)^{-1}\rt m])\ast \g(m^r, \g(m,m^r)^{-1}\rt m) \ast \s(\g(m,m^r)^{-1}\ast \eta, m) \ast h
\end{align*}
where we have used \eqref{right-action-gr-on-multp} and \eqref{right-action-theta-comp-gr} on the first equality, and \eqref{right-action-theta-comp-gr} on the second equality. The result thus follows.
\end{proof}

\begin{remark}
We note that in case \eqref{theta-map-gr} to be trivial, the adjoint action \eqref{Ad-cocycle-double-cross-prod} coincides with that of \cite[(2.14)]{EsenSutl17}.
\end{remark}

\begin{remark}
Let us note also that the derivative of \eqref{Ad-cocycle-double-cross-prod} yields
\begin{equation}\label{derv-of-gamma}
\t(v,w) = \dt\,\ds\, \g(v_t,m_s) - \dt\,\ds\, \g(m_s,v_t) \in \G{h},
\end{equation}
for the Lie algebra twisted cocycle $\t:T_oM\otimes T_oM \to \G{h}$. Indeed, it suffices to differentiate \eqref{Ad-group} on two curves $(m_s,1),(v_t,1) \in M\bowtie_\g H$, with $m_0 = o \in M$ and $m'_0 = w\in T_oM$, and, $v_0 = o \in M$ and $v'_0 = v\in T_oM$.
\end{remark}

\begin{proposition}
The tangent group $T(M\bowtie_\g H)$ of the cocycle double cross product $M\bowtie_\g H$ has the structure of a cocycle double cross product. More precisely,
\[
T(M\bowtie_\g H) \cong (M\times T_oM) \bowtie_\Lambda TH
\]
as Lie groups, where $\Lambda:(M\times T_oM)\times (M\times T_oM) \to TH$ is given by
\begin{align}\label{Lambda-on-TMxTM}
\begin{split}
&\Lambda\big((m_1,v_1), (m_2,v_2)\big) = \\
& \Big(\g(m_1,m_2), TR_{\s(\g(m_2,m_2^r)^{-1}, m_2)}\big(\g(m_2^r,\g(m_2,m_2^r)^{-1}\rt T\G{R}_{m_2}(v_1))+ \g(m_2^r,\g(m_2,m_2^r)^{-1} \rt m_2)\ast \\
& \hspace{10cm} \g(\g(m_2,m_2^r)^{-1}\rt v_1,\g(m_2,m_2^r)^{-1}\rt m_2)\big) + \\
&TL_{\g(m_2^r,\g(m_2,m_2^r)^{-1}\rt m_2)}(\s(\s(\g(m_2,m_2^r)^{-1},v_1),m_2))\Big)
\end{split}
\end{align}
for any $(m_1,v_1), (m_2,v_2) \in M \times T_oM$.
\end{proposition}

\begin{proof}
The (left) trivialization 
\begin{equation}\label{left-triv-TMH}
T(M\bowtie_\g H) \cong (M\bowtie_\g H) \ltimes (T_oM\bowtie_{\t} \G{h})
\end{equation}
of the tangent group $T(M\bowtie_\g H)$ enables us to set
\begin{equation}\label{TH-into-T(MxH)}
j:TH \cong H\ltimes \G{h} \to T(M\bowtie_\g H) \cong (M\bowtie_\g H) \ltimes (T_oM\bowtie_{\t} \G{h}), \qquad (h,\eta) \mapsto \Big((o,h), (0,\eta)\Big).
\end{equation}
We then observe for any $(h_1,\eta_1),(h_2,\eta_2) \in TH \cong H\ltimes \G{h}$ that
\begin{align*}
& j(h_1,\eta_1)j(h_2,\eta_2) = \Big((o,h_1), (0,\eta_1)\Big)\Big((o,h_2), (0,\eta_2)\Big) = \Big((o,h_1\ast h_2), \Ad_{(o,h_2)^{-1}}(0,\eta_1) + (0,\eta_2)\Big) = \\
&\Big((o,h_1\ast h_2), (0,\eta_2+\Ad_{{h_2}^{-1}}\eta_1)\Big) = j((h_1,\eta_1)(h_2,\eta_2)),
\end{align*}
where we used \eqref{Ad-cocycle-double-cross-prod} on the third equality. As such, the inclusion \eqref{TH-into-T(MxH)} being a group homomorphism $TH$ may be viewed as a subgroup of $T(M\bowtie_\g H)$. Moreover, in view of the trivialization \eqref{left-triv-TMH} we have
\begin{equation}\label{TMH-trivialization}
T(M\bowtie_\g H) \cong (M\times T_oM) \times TH
\end{equation}
as manifolds. We may thus conclude by Proposition \ref{prop-universal-gr} that $T(M\bowtie_\g H)$ has the structure of a cocycle double cross product Lie group for some $\Lambda:(M\times T_oM)\times (M\times T_oM) \to TH$.

Furthermore, keeping \eqref{Ad-cocycle-double-cross-prod} and \eqref{TMH-trivialization} in mind, given any $(m_1,v_1),(m_2,v_2) \in M \times T_oM$ we have on one hand
\begin{align*}
& \Big((m_1,1), (v_1,0)\Big)\Big((m_2,1), (v_2,0)\Big) = \Big((m_1\cdot m_2, \g(m_1,m_2)),(v_2,0)+\Ad_{(m_2,1)^{-1}}(v_1,0)\Big) = \\
& \Big((m_1\cdot m_2, \g(m_1,m_2)),(v_2 + T\G{L}_{m_2^r}\big(\g(m_2,m_2^r)^{-1}\rt T\G{R}_{m_2}(v_1)\big), \\
& TR_{\s(\g(m_2,m_2^r)^{-1}, m_2)}\big(\g(m_2^r,\g(m_2,m_2^r)^{-1}\rt (v_1\cdot m_2))+ \g(m_2^r,\g(m_2,m_2^r)^{-1} \rt m_2)\ast \\
&\hspace{8cm} \g(\g(m_2,m_2^r)^{-1}\rt v_1,\g(m_2,m_2^r)^{-1}\rt m_2)\big) +\\
& TL_{\g(m_2^r,\g(m_2,m_2^r)^{-1}\rt m_2)}(\s(\s(\g(m_2,m_2^r)^{-1},v_1),m_2)))\Big),
\end{align*}
while on the other hand
\begin{align*}
& \Big((m_1,1), (v_1,0)\Big)\Big((m_2,1), (v_2,0)\Big) = (m_1,v_1)(m_2,v_2) = \\
& \Big((m_1,v_1)\cdot (m_2,v_2), \Lambda\big((m_1,v_1), (m_2,v_2)\big)\Big) \in (M\times T_oM) \bowtie_\Lambda TH.
\end{align*}
Accordingly, we conclude
\begin{align*}
(m_1,v_1)\cdot (m_2,v_2) = \Big(m_1\cdot m_2, v_2 +T\G{L}_{m_2^r}\big(\g(m_2,m_2^r)^{-1}\rt T\G{R}_{m_2}(v_1)\big)\Big) \in M\times T_oM,
\end{align*}
as well as \eqref{Lambda-on-TMxTM}.
\end{proof}

\begin{remark}
Running a similar calculation, we may derive the left action $\btr:TH \times (M\times T_oM) \to (M\times T_oM)$ of $TH$ on $M\times T_oM$, and the analogue $\Sigma:TH \times (M\times T_oM) \to TH$ of the mapping \eqref{right-action-gr}.  This time we use \eqref{Ad-cocycle-double-cross-prod} with $(m,v) \in TM \cong M\times T_oM$ and $(h,\eta)\in TH\cong H\times \G{h}$ to calculate
\[
\Big((o,h),(0,\eta)\Big)\Big((m,1),(v,0)\Big) = \Big((h\rt m,\s(h,m)),(v,0) + \Ad_{(m,1)^{-1}}(0,\eta)\Big).
\]
As a result, we obtain
\begin{align*}
&\btr:TH \times (M\times T_oM) \to TM, \\
& (h,\eta) \btr (m,v) = \Big(h\rt m, v+T\G{L}_{m^r}\big(TL_{\g(m,m^r)^{-1}}(\eta) \rt m\big)\Big),
\end{align*}
and
\begin{align*}
&\Sigma:TH \times (M\times T_oM) \to TH, \\
& \Sigma\big((h,\eta), (m,v)\big) = \\
& \Big(\s(h,m), TR_{\s(\g(m,m^r)^{-1},m)}\g(m^r,TL_{\g(m,m^r)^{-1}} (\eta) \rt m) + TL_{\g(m^r,\g(m,m^r)^{-1}\rt m)}\s(TL_{\g(m,m^r)^{-1}} (\eta) ,m)\Big).
\end{align*}
\end{remark}

\subsection{Euler-Lagrange equations on a cocycle double cross product}~

We shall now present the Euler-Lagrange equations \eqref{EL-eqn} in the case $G=M\bowtie_\g  H$. To this end, we shall derive next the cotangent lift of the left translation over a cocycle double cross product. The coadjoint action of a cocycle double cross product Lie group on its Lie algebra will follow afterwards.  

\begin{proposition}
Given any $(m_1,h_1),(m_2,h_2)\in M\bowtie_\g  H$, any $\a:=\alpha_{m_1\cdot(h_1 \rt m_2)}\in T^{\ast}_{m_1\cdot (h_1 \rt m_2)} M$, and any $ \b:=\beta_{\gamma(m_1,h_1 \rt m_2)\ast\s(h_1,m_2)\ast h_2}\in T^{\ast}_{\gamma(m_1,h_1 \rt m_2)\ast\s (h_1,m_2)\ast h_2} H$, we have
\begin{align*}
& T^{\ast}_{(m_2,h_2)}L_{(m_1,h_1)}(\alpha,\beta)= \\
&\hspace{.7cm}\Big((T^\ast\G{L}_{m_1}\alpha)\overset{\ast}{\lt } h_1+(T^\ast\left(\gamma_{m_1} \circ R_{\s(h_1,m_2)\ast h_2}\right)  \beta)\overset{\ast}{\vartriangleleft} h_1+T^{\ast}\left(\s_{h_1}\circ L_{\gamma(m_1,h_1 \rt m_2)}\circ R_{h_2}\right)\beta, \\
&T^{\ast}L_{\gamma(m_1,h_1 \rt m_2)\ast\s(h_{1},m_2)}\beta \Big),
\end{align*}  
where
\begin{align*}
& \G{L}_m:M\to M, \qquad n \mapsto m\cdot n, \\
&\g_m: M\to H, \qquad n\mapsto \g(m,n), \\
& \s_h: M\to H, \qquad m\mapsto \s(h,m).
\end{align*}
\end{proposition}

\begin{proof}
Given any $U := U_{m_2}\in T_{m_2}M$, and any $S := S_{h_2} \in T_{h_2}H$, the claim follows from the pairing
\[
\langle T^{\ast}_{(m_2,h_2)}L_{(m_1,h_1)}(\alpha,\beta), (U,S)\rangle = \langle (\alpha,\beta), TL_{(m_1,h_1)}(U,S)\rangle.
\]
\end{proof}

Accordingly, in particular for $(m_2,h_2)=(o,1)$, we conclude the following. 

\begin{corollary}
Given any $\a:=\alpha_m\in T^{\ast}_m M$, and any $ \b:=\beta_h\in T^{\ast}_h H$,
\begin{equation}\label{T^*(e,1)}
T^{\ast}L_{(m,h)}(\alpha,\beta)=\Big(
  (T^{\ast}\G{L}_m\alpha )\overset{\ast}{\lt} h+(T^{\ast}\left( \gamma_m 
\circ R_{h}\right)  \beta )\overset{\ast}{\lt} h+T^{\ast} \s_{h} 
\beta, T^{\ast}L_{h}\beta\Big).
\end{equation}  
\end{corollary}

As for the coadjoint action, we proceed as follows. 

\begin{proposition}
Given any $v\in T_o M$, any $\eta\in \mathfrak{h}$, any $\a=\alpha_o\in T^{\ast}_oM $, and any $\b \in \mathfrak{h}^{\ast}$, 
 \begin{equation}\label{coadact}
 \ad^{\ast}_{(v,\eta)}(\alpha, \b)=\left(\G{ad}^\ast_v(\alpha)-\alpha\overset{\ast}{\lt} \eta -\G{a}_{\eta}^{\ast}(\b)- \theta^{\ast}_{v}(\b) , \ad_{\eta}^{\ast}\b+\mathfrak{b}_{v}^{\ast}(\alpha)+{}_{v}\psi^{\ast}(\b)
\right ),
\end{equation}
 where
 \begin{align*}
 &\G{ad}_v:T_oM\to T_oM, \qquad w\mapsto [[v,w]],\\
 & \mathfrak{b}_{v}:\G{h}\to T_oM, \qquad \eta\mapsto \eta\rt v, \\
 & \G{a}_{\eta}:T_oM\to \G{h}, \qquad v\mapsto \psi(\eta,v), \\
 & {}_{v}\psi:\G{h}\to \G{h}, \qquad \eta\mapsto \psi(\eta,v), \\
 & \t_v: T_oM\to \G{h}, \qquad w \mapsto \t(v,w).
 \end{align*}
 \end{proposition}

\begin{proof}
The claim follows at once from \eqref{cocycle-double-cross-sum-bracket}, and 
\[
\langle \ad^{\ast}_{(v_1,\eta_1)}(\a,\b), (v_2,\eta_2)\rangle=-\langle(\a,\b),[(v_1,\eta_1),(v_2,\eta_2)]\rangle.
\]
\end{proof}
 
We are now ready to present the Euler-Lagrange equations for a Lagrangian on the tangent group of a cocycle double cross product group.

\begin{proposition}
Given a Lagrangian $\C{L}=\C{L}\left(m,h,v,\eta\right)$
defined on $(M\bowtie_\g  H)\ltimes(T_o M\bowtie_{\t} \mathfrak{h})$, the (trivialized) Euler-Lagrange equations are given by
\begin{align}
&\frac{d}{dt}\frac{\delta\C{L}}{\delta v}   =  
(T^{\ast}\G{L}_m\left(\frac{\delta\C{L}}{\delta m}\right) )\overset{\ast}{\lt} h+(T^{\ast}\left( \gamma_m 
\circ R_{h}\right)\left(  \frac{\delta\C{L}}{\delta h}\right))\overset{\ast}{\lt} h+ \label{mEL}\\
& \hspace{4cm} T^{\ast} \s_{h} \left(  \frac{\delta\C{L}}{\delta h}\right) -\G{ad}^\ast_v\left(\frac{\delta\C{L}}{\delta v}\right)+\left(\frac{\delta\C{L}}{\delta v}\right)\overset{\ast}{\lt} \eta +\G{a}_{\eta}^{\ast}\left(  \frac{\delta\C{L}}{\delta \eta}\right)+\theta^{\ast}_{v}\left(  \frac{\delta\C{L}}{\delta \eta}\right),\notag\\
&\frac{d}{dt}\frac{\delta\C{L}}{\delta\eta}   =T^{\ast}L_h\left(  \frac{\delta\C{L}}{\delta h}\right)  -\ad_{\eta}^{\ast}\left(\frac{\delta\C{L}}{\delta\eta}\right)-\mathfrak{b}_v^{\ast}\left(\frac
{\delta\C{L}}{\delta v}\right) - {}_{v}\psi^{\ast}\left(\frac{\delta\C{L}}{\delta h}\right). \label{mEL-II}
\end{align}
\end{proposition}

\begin{proof}
It follows from \eqref{EL-eqn} that the (trivialized) Euler-Lagrange equations are given by
\[
\frac{d}{dt}\left(\frac{\delta\C{L}}{\delta v},\frac{\delta
\C{L}}{\delta\eta}\right)  = T^\ast L_{(m,h)}\left(\frac{\delta\C{L}}{\delta m},\frac{\delta\C{L}}{\delta h}\right) -\ad_{\left(v,\eta\right)
}^{\ast}\left(  \frac{\delta\C{L}}{\delta v},\frac{\delta\C{L}}{\delta\eta}\right).
\]
Then, on the one hand \eqref{T^*(e,1)} yields
\begin{align*}
& T^\ast L_{(m,h)}\left(\frac{\delta\C{L}}{\delta m},\frac{\delta\C{L}}{\delta h}\right) = \Big((T^{\ast}\G{L}_m\left(\frac{\delta\C{L}}{\delta m}\right) )\overset{\ast}{\lt} h+(T^{\ast}\left( \gamma_m 
\circ R_{h}\right)\left(  \frac{\delta\C{L}}{\delta h}\right))\overset{\ast}{\lt} h+T^{\ast} \s_{h} \left(  \frac{\delta\C{L}}{\delta h}\right), T^{\ast}L_{h}\left(  \frac{\delta\C{L}}{\delta h}\right)\Big),
\end{align*}
on the other hand it follows from \eqref{coadact} that
\begin{align*}
& \ad_{\left(v,\eta\right)
}^{\ast}\left(  \frac{\delta\C{L}}{\delta v},\frac{\delta\C{L}}{\delta\eta}\right) = \\
& \left(\G{ad}^\ast_v\left(\frac{\delta\C{L}}{\delta v}\right)-\left(\frac{\delta\C{L}}{\delta v}\right)\overset{\ast}{\lt} \eta -\G{a}_{\eta}^{\ast}\left(  \frac{\delta\C{L}}{\delta \eta}\right)- \theta^{\ast}_{v}\left(  \frac{\delta\C{L}}{\delta \eta}\right) , \ad_{\eta}^{\ast}\left(  \frac{\delta\C{L}}{\delta \eta}\right)+\mathfrak{b}_{v}^{\ast}\left(\frac{\delta\C{L}}{\delta v}\right)+{}_{v}\psi^{\ast}\left(  \frac{\delta\C{L}}{\delta \eta}\right)
\right ). 
\end{align*}
The claim, thus, follows.
\end{proof}

The reduction of $T(M\bowtie_\g  H) \cong (M\times T_oM) \times TH$ by $TH$, hence, yields Euler-Lagrange type equations on $TM \cong T(M\bowtie_\g  H) / TH \cong M\times T_oM$. Explicitly, we have the following.

\begin{corollary}
The Euler-Lagrange type equations on $TM \cong M\times T_o M$, associated to a (reduced) Lagrangian $\C{L}=\C{L}\left(m,v\right)$ on $M\times T_o M$ are given by
\begin{equation}
\frac{d}{dt}\frac{\delta\C{L}}{\delta v}   =  
T^{\ast}\G{L}_m\left(\frac{\delta\C{L}}{\delta m}\right)  -\G{ad}^\ast_v\left(\frac{\delta\C{L}}{\delta v}\right).
\end{equation}
\end{corollary}

On the other extreme, the reduction of $T(M\bowtie_\g  H) \cong (M\bowtie_\g  H) \ltimes (T_oM \bowtie_\t \G{h})$ by $M\bowtie_\g  H$ gives rise to the Euler-Poincar\'e equations on $T_oM \bowtie_\t \G{h}$. More precisely, we have the following.

\begin{corollary}
Given a (reduced) Lagrangian $\ell=\ell(v,\eta)$ on $T_oM\bowtie_\t \G{h}$, the Euler-Poincar\'{e} equations are given by 
 \begin{align}
& \frac{d}{dt}\frac{\delta\ell}{\delta v}   =  -\G{ad}^\ast_v\left(\frac{\delta\ell}{\delta v}\right)+\left(\frac{\delta\ell}{\delta v}\right)\overset{\ast}{\lt} \eta +\G{a}_{\eta}^{\ast}\left(  \frac{\delta\ell}{\delta \eta}\right)+\theta^{\ast}_{v}\left(  \frac{\delta\ell}{\delta \eta}\right)\label{EPE},\\
&\frac{d}{dt}\frac{\delta\ell}{\delta\eta}   =-\ad_{\eta}^{\ast}\left(\frac{\delta\ell}{\delta\eta}\right)-\mathfrak{b}_v^{\ast}\left(\frac
{\delta\ell}{\delta v}\right) - {}_{v}\psi^{\ast}\left(\frac{\delta\ell}{\delta \eta}\right).\label{EPE-II}
\end{align}
\end{corollary}

\section{Hamiltonian Dynamics on Cocycle Double Cross Product Lie Groups}

\subsection{The cotangent group of a cocycle double cross product}~

Let us now investigate the structure of the cotangent group $T^\ast G$ for $G:=M\bowtie_\g H$. To this end, we shall make use of the (right) trivialization $tr_{T^\ast G}^R:T^\ast G\to \G{g}^\ast \rtimes G$, where $\G{g}^\ast = T_o^\ast M \oplus \G{h}^\ast$.

\begin{proposition}
The cotangent group $T^\ast (M\bowtie_\g H)$ of the cocycle double cross product $M\bowtie_\g H$ has the structure of a cocycle double cross product. More precisely,
\[
T^\ast(M\bowtie_\g H) \cong (T^\ast_oM\times M) \bowtie_\Upsilon T^\ast H
\]
as Lie groups, where $\Upsilon:(T^\ast_oM\times M)\times (T^\ast_oM\times M) \to T^\ast H$ is given by
\begin{equation}\label{Upsilon-on-T-star-MxT-star-M}
\Upsilon\big((\kappa_1,m_1), (\kappa_2,m_2)\big) =  \Big(T^\ast \big(L_{\g({m_1},{m_1}^r)^{-1}}\circ {}_{m_1}\vp \circ \G{L}_{m_1^r}\big)(\kappa_2), \g(m_1,m_2)\Big)
\end{equation}
for any $(\kappa_1,m_1), (\kappa_2,m_2) \in T^\ast_oM \times M$, where
\begin{align*}
& {}_m\vp: H\to M, \qquad h\mapsto h\rt m, \\
& \G{L}_m:M\to M, \qquad n\mapsto m\cdot n.
\end{align*}
\end{proposition}

\begin{proof}
In view of the (right) trivialization 
\begin{equation}\label{right-triv-T-star-MH}
T^\ast(M\bowtie_\g H) \cong  (T^\ast_oM\oplus \G{h}^\ast) \rtimes (M\bowtie_\g H)
\end{equation}
of the cotangent group $T^\ast(M\bowtie_\g H)$, we set
\begin{equation}\label{T-star-H-into-T-star-(MxH)}
j:T^\ast H \cong \G{h}^\ast\rtimes H \to T^\ast(M\bowtie_\g H) \cong  (T^\ast_oM\oplus \G{h}^\ast)\rtimes (M\bowtie_\g H), \qquad (\lambda,h) \mapsto \Big((0,\lambda), (o,h)\Big).
\end{equation}
Then, for any $(\lambda_1,h_1),(\lambda_2,h_2) \in T^\ast H \cong \G{h}^\ast \rtimes H$,
\begin{align*}
& j(\lambda_1,h_1)j(\lambda_2,h_2) = \Big((0,\lambda_1), (o,h_1)\Big)\Big((0,\lambda_2), (o,h_2)\Big) = \Big((0,\lambda_1)+\Ad^\ast_{(o,h_1)}(0,\lambda_2) , (o,h_1\ast h_2)\Big) = \\
&\Big((0,\lambda_1+\Ad^\ast_{h_1}\lambda_2), (o,h_1\ast h_2)\Big) = j((\lambda_1,h_1)(\lambda_2,h_2)),
\end{align*}
using 
\[
\langle \Ad^\ast_h \lambda, \eta\rangle = \langle \lambda, \Ad_{h^{-1}}\eta\rangle
\]
on the third equality. That is, the inclusion \eqref{T-star-H-into-T-star-(MxH)} is a group homomorphism, via which $T^\ast H$ may be realised as a subgroup of $T^\ast(M\bowtie_\g H)$. Moreover, in view of the trivialization \eqref{right-triv-T-star-MH} we have
\begin{equation}\label{T-star-MH-trivialization}
T^\ast(M\bowtie_\g H) \cong (T^\ast_oM\times M) \times T^\ast H
\end{equation}
as manifolds. As a result, it follows from Proposition \ref{prop-universal-gr} that $T^\ast(M\bowtie_\g H)$ is a cocycle double cross product for some $\Upsilon:(T^\ast_oM\times M)\times (T^\ast_oM\times M) \to T^\ast H$.

We shall now present the twisted cocycle $\Upsilon:(T^\ast_oM\times M)\times (T^\ast_oM\times M) \to T^\ast H$ explicitly. To this end, given any $(\kappa_1,m_1),(\kappa_2,m_2) \in T^\ast_oM \times M$ we first recall
\[
\Big((\kappa_1,0), (m_1,1)\Big)\Big((\kappa_2,0), (m_2,1)\Big) = \Big((\kappa_1,0)+\Ad^\ast_{(m_1,1)}(\kappa_2,0), (m_1\cdot m_2, \g(m_1,m_2))\Big).
\]
We then note from \eqref{Ad-cocycle-double-cross-prod} that, given any $(v,\eta) \in T_oM\oplus \G{h}$ and any $(\kappa,m)\in T^\ast_oM\times M$,
\begin{align*}
&\langle \Ad^\ast_{(m,1)}(\kappa,0), (v,\eta) \rangle = (\kappa,0),\Ad_{(m,1)^{-1}} (v,\eta) \rangle = \\
&  \langle \kappa, T\G{L}_{m^r} \big[\g(m,m^r)^{-1} \rt T\G{R}_m v+TL_{\g(m,m^r)^{-1}}(\eta)\rt m)\big]\rangle =\\
& \langle \Big(T^\ast\G{R}_m\big[\g(m,m^r) \overset{\ast}{\rt} T^\ast\G{L}_{m^r}(\kappa)\big], T^\ast \big(L_{\g(m,m^r)^{-1}}\circ {}_m\vp \circ \G{L}_{m^r}\big)(\kappa)\Big), (v,\eta)\rangle.
\end{align*}
Accordingly,
\begin{align*}
& \Big((\kappa_1,0), (m_1,1)\Big)\Big((\kappa_2,0), (m_2,1)\Big) = \\
& \Big(\big(\kappa_1 + T^\ast\G{R}_{m_1}\big[\g({m_1},{m_1}^r) \overset{\ast}{\rt} T^\ast\G{L}_{m_1^r}(\kappa_2)\big], T^\ast \big(L_{\g({m_1},{m_1}^r)^{-1}}\circ {}_{m_1}\vp \circ \G{L}_{m_1^r}\big)(\kappa_2)\big), (m_1\cdot m_2, \g(m_1,m_2))\Big). 
\end{align*}
On the other hand
\begin{align*}
& \Big((\kappa_1,0), (m_1,1)\Big)\Big((\kappa_2,0), (m_2,1)\Big) = (\kappa_1,m_1)(\kappa_2,m_2) = \\
& \Big((\kappa_1,m_1)\cdot (\kappa_2,m_2), \Upsilon\big((\kappa_1,m_1), (\kappa_2,m_2)\big)\Big) \in (T^\ast_oM\times M) \bowtie_\Upsilon T^\ast H.
\end{align*}
We then conclude both \eqref{Upsilon-on-T-star-MxT-star-M} and
\begin{align*}
(m_1,v_1)\cdot (m_2,v_2) = \Big(\kappa_1 + T^\ast\G{R}_{m_1}\big[\g({m_1},{m_1}^r) \overset{\ast}{\rt} T^\ast\G{L}_{m_1^r}(\kappa_2)\big], m_1\cdot m_2\Big) \in T^\ast_oM\times M,
\end{align*}
\end{proof}

\begin{remark}
The left action $\btr:T^\ast H \times (T^\ast_oM\times M) \to (T^\ast_oM\times M)$ of $T^\ast H$ onto $T^\ast_oM\times M$, and the mapping $\Sigma:T^\ast H \times (T^\ast_oM\times M) \to T^\ast H$ may be obtained as above. Indeed, given any $(\kappa,m) \in T^\ast_oM\times M$, and any $(\lambda,h) \in T^\ast H$, in the group $(T^\ast_oM\times M) \bowtie_\Upsilon T^\ast H$ we have
\begin{align*}
& (\lambda,h)(\kappa,m) = \Big((0,\lambda),(o,h)\Big)\Big((\kappa,0),(m,1)\Big) = \Big((0,\lambda) + \Ad^\ast_{(o,h)}(\kappa,0), (h\rt m, \s(h,m))\Big) = \\
& \Big((h\overset{\ast}{\rt} \kappa,\lambda) , (h\rt m, \s(h,m))\Big).
\end{align*}
Hence we obtain
\begin{align*}
& \btr:T^\ast H \times (T^\ast_oM\times M) \to (T^\ast_oM\times M), \\
& (\lambda,h)\btr (\kappa,m) = (h\overset{\ast}{\rt} \kappa,h\rt m),
\end{align*}
and
\begin{align*}
& \Sigma:T^\ast H \times (T^\ast_oM\times M) \to T^\ast H, \\
& \Sigma\big((\lambda,h), (\kappa,m)\big) = (\lambda,\s(h, m)),
\end{align*}
\end{remark}

\subsection{Hamilton's equations on a cocycle double cross product}~

Following \cite{EsenSutl16}, let us note that the Hamilton's equations for a (trivialized) Hamiltonian function $\C{H}:\G{g}^\ast \rtimes G \to \mathbb{R}$, $\C{H}=\C{H}(\mu,g)$, are given by
\begin{align}\label{HamiltonEqns}
\begin{split}
& \frac{d\mu}{dt} = -T^\ast R_g\left(\frac{\d \C{H}}{\d g}\right) + \ad^\ast_{\frac{\d \C{H}}{\d \mu}}(\mu), \\
& \frac{dg}{dt} = TR_g\left(\frac{\d \C{H}}{\d \mu}\right).
\end{split}
\end{align}

Therefore, we need the cotangent lift of the right translation, which we shall derive below. 

\begin{proposition}
Given any $(m_1,h_1),(m_2,h_2)\in M\bowtie_\g  H$, any $\a:=\alpha_{m_2\cdot(h_2 \rt m_1)}\in T^{\ast}_{m_2\cdot (h_2 \rt m_1)} M$, and any $ \b:=\beta_{\gamma(m_2,h_2 \rt m_1)\ast\s(h_2,m_1)\ast h_1}\in T^{\ast}_{\gamma(m_2,h_2 \rt m_1)\ast\s (h_2,m_1)\ast h_1} H$, we have
\begin{align*}
& T^{\ast}_{(m_2,h_2)}R_{(m_1,h_1)}(\alpha,\beta)= \\
&\hspace{.7cm}\Big(T^\ast \G{R}_{h_2\rt m_1} \a + T^\ast({}_{h_2\rt m_1}\g \circ R_{(h_2\lt m_1)\ast h_1})\b, \\
&T^\ast({}_{m_1}\vp \circ L_{m_2})\a + T^\ast({}_{m_1}\vp \circ \g_{m_2}\circ R_{(h_2\lt m_1)\ast h_1})\b + T^\ast({}_{m_1}\s\circ R_{h_1}\circ L_{\g(m_2,h_2\rt m_1)})\b \Big),
\end{align*}  
where
\begin{align*}
& \G{R}_m:M\to M, \qquad n \mapsto n\cdot m, \\
&\g_m: M\to H, \qquad n\mapsto \g(m,n), \\
&{}_m\g: M\to H, \qquad n\mapsto \g(n,m), \\
& {}_m\vp: H\to M, \qquad h\mapsto h\rt m, \\
& {}_m\s: H\to H, \qquad h\mapsto \s(h,m).
\end{align*}
\end{proposition}

\begin{proof}
Given any $U := U_{m_2}\in T_{m_2}M$, and any $S := S_{h_2} \in T_{h_2}H$, once again the claim follows from the pairing
\[
\langle T^{\ast}_{(m_2,h_2)}R_{(m_1,h_1)}(\alpha,\beta), (U,S)\rangle = \langle (\alpha,\beta), T_{(m_2,h_2)}R_{(m_1,h_1)}(U,S)\rangle.
\]
\end{proof}

In particular, for $(m_2,h_2)=(o,1)$, we arrive at the following. 

\begin{corollary}
Given any $\a:=\alpha_m\in T^{\ast}_m M$, and any $ \b:=\beta_h\in T^{\ast}_h H$,
\begin{equation}\label{T^*(e,1)-right}
T^{\ast}R_{(m,h)}(\alpha,\beta)=\Big(T^\ast \G{R}_m \a + T^\ast({}_m\g \circ R_h)\b, T^\ast({}_m\vp)\a + T^\ast({}_m\s\circ R_h)\b \Big).
\end{equation}  
\end{corollary}

We are, thus, all set to derive the Hamilton's equations.

\begin{proposition}
Given a Hamiltonian $\C{H}=\C{H}\left(\kappa,\lambda,m,h\right)$
defined on $(T_o M\bowtie_{\t} \mathfrak{h})^\ast \rtimes (M\bowtie_\g  H)$, the (trivialized) Hamilton's equations are given by
\begin{align}
&\frac{d \kappa}{dt} =-T^\ast \G{R}_m \left(\frac{\delta\mathcal{H}}{\delta m}\right) - T^\ast({}_m\g \circ R_h)\left(\frac{\delta\mathcal{H}}{\delta h}\right) + \G{ad}^\ast_{\frac{\delta\mathcal{H}}{\delta \kappa}}(\kappa)-\kappa\overset{\ast}{\lt} \left(\frac{\delta\mathcal{H}}{\delta \lambda}\right) -\G{a}_{\frac{\delta\mathcal{H}}{\delta \lambda}}^{\ast}(\lambda)- \theta^{\ast}_{\frac{\delta\mathcal{H}}{\delta \kappa}}(\lambda), \label{HamEq-I}\\
& \frac{d \lambda}{dt} =-T^\ast({}_m\vp)\left(\frac{\delta\mathcal{H}}{\delta m}\right) - T^\ast({}_m\s\circ R_h)\left(\frac{\delta\mathcal{H}}{\delta h}\right) +\ad_{\frac{\delta\mathcal{H}}{\delta \lambda}}^{\ast}\lambda+\mathfrak{b}_{\frac{\delta\mathcal{H}}{\delta \kappa}}^{\ast}(\kappa)+{}_{\frac{\delta\mathcal{H}}{\delta \kappa}}\psi^{\ast}(\lambda),  \label{HamEq-II}\\
&\frac{dm}{dt} =T \G{R}_m\left(\frac{\delta\mathcal{H}}{\delta\kappa}\right) + \left(\frac{\delta\mathcal{H}}{\delta\lambda}\right)\rt m, \label{HamEq-III} \\
& \frac{dh}{dt} =TR_h[T({}_m\gamma)\left(\frac{\delta\mathcal{H}}{\delta\kappa}\right) +T({}_m\s)\left(\frac{\delta\mathcal{H}}{\delta\lambda}\right)]. \label{HamEq-IV}
\end{align}
\end{proposition}

\begin{proof}
Given $\mathcal{H}=\mathcal{H}(\kappa,\lambda,m,h)$, it follows from \eqref{HamiltonEqns} that the Hamilton's equations are
\begin{align*}
&\frac{d}{dt}(\kappa ,\lambda) =-T^\ast R_{(m,h)}\left(\frac{\delta\mathcal{H}}{\delta m}, \frac{\delta\mathcal{H}}{\delta h}\right) + \ad^{\ast}_{(\frac{\delta\mathcal{H}}{\delta\kappa}, \frac{\delta\mathcal{H}}{\delta\lambda})}(\kappa,\lambda), \\
&\frac{d}{dt} (m,h)=TR_{(m,h)}\left(\frac{\delta\mathcal{H}}{\delta\kappa}, \frac{\delta\mathcal{H}}{\delta\lambda}\right).
\end{align*}
As for the former, we substitute \eqref{coadact} and \eqref{T^*(e,1)-right} to get \eqref{HamEq-I} and \eqref{HamEq-II}. We next use
\[
TR_{(m,h)}(U,S) = \Big(T \G{R}_m(U) + S\rt m, TR_h\big[T({}_m\gamma)(U) +T({}_m\s)(S)\big]\Big)
\]
given any $U := U_o\in T_oM$, and any $S := S_1 \in T_1H$, to obtain \eqref{HamEq-III} and \eqref{HamEq-IV}.
\end{proof}
 
Along the lines of the previous subsection we proceed with the reduction of $T^\ast(M\bowtie_\g  H) \cong (T^\ast_oM \times M) \times T^\ast H$ by $T^\ast H$, which yields the Hamilton type equations on the cotangent bundle $T^\ast M \cong T^\ast(M\bowtie_\g  H) / T^\ast H \cong T^\ast_oM \times M$ of the quotient $M=G/H$. More precisely, we have the following.

\begin{corollary}
The Hamilton type equations on $T^\ast M \cong T^\ast_oM \times M$, associated to a (reduced) Hamilton function $\C{H}=\C{H}\left(\kappa,m\right)$ on $T^\ast_oM \times M$ are given by
\begin{align}
&\frac{d \kappa}{dt} =-T^\ast \G{R}_m \left(\frac{\delta\mathcal{H}}{\delta m}\right)  + \G{ad}^\ast_{\frac{\delta\mathcal{H}}{\delta \kappa}}(\kappa), \\
&\frac{dm}{dt} =T \G{R}_m\left(\frac{\delta\mathcal{H}}{\delta\kappa}\right) . 
\end{align}
\end{corollary}

We next aim to present the Lie-Poisson equations on the dual space $T^\ast_oM \oplus \G{h}^\ast$ through the reduction of $T^\ast(M\bowtie_\g  H) \cong (T^\ast_oM \oplus \G{h}^\ast)\rtimes (M\bowtie_\g  H)$ by $M\bowtie_\g  H$.

\begin{corollary}
Given a (reduced) Hamiltonian $\G{H}=\G{H}(\kappa,\lambda)$ on $T^\ast_oM\oplus \G{h}^\ast$, the Lie-Poisson equations are given by 
 \begin{align}
&\frac{d \kappa}{dt} =  \G{ad}^\ast_{\frac{\delta\G{H}}{\delta \kappa}}(\kappa)-\kappa\overset{\ast}{\lt} \left(\frac{\delta\G{H}}{\delta \lambda}\right) -\G{a}_{\frac{\delta\G{H}}{\delta \lambda}}^{\ast}(\lambda)- \theta^{\ast}_{\frac{\delta\G{H}}{\delta \kappa}}(\lambda), \label{LiePoissonEqn-I}\\
& \frac{d \lambda}{dt} =\ad_{\frac{\delta\G{H}}{\delta \lambda}}^{\ast}\lambda+\mathfrak{b}_{\frac{\delta\G{H}}{\delta \kappa}}^{\ast}(\kappa)+{}_{\frac{\delta\G{H}}{\delta \kappa}}\psi^{\ast}(\lambda) . \label{LiePoissonEqn-II}
\end{align}
\end{corollary}

\section{Higher Order Dynamics Through Cocycle Double Cross Product Lie Groups}

\subsection{$n$th tangent group of a Lie group}~

Given a manifold $Q$, the $n$th order tangent bundle $\tau^n:T^nQ\to Q$ is defined to be the equivalence classes of curves which identifies two curves with the same $n$th order derivatives at 0, see for instance \cite{abrunheiro2011cubic, colombo2011geometry, colombo2013optimal, colombo2014unified, gay2012invariant,gay2011higher}, where the bundle projection being
\[
\tau^n:T^nQ\to Q, \qquad \tau^n([p]) := p(0)
\]
given the curve $p(t)\in Q$, of class $[p]$. 

Just like the first order tangent group of a Lie group, the $n$th order tangent bundle $T^nG$ of a Lie group $G$ has the structure of a Lie group, see for instance \cite{gay2012invariant,KolaMichSlov-book,Vizm13}, with the Lie group structure given by
\[
[p][q] := [pq].	
\] 
More precisely, given a curve $p(t)\in G$, let $\d^\ell p := p^{-1}p'$ denotes the \emph{left logarithmic derivative}. Accordingly, 
\[
tr^L_{T^nG}:T^nG\to G\times \G{g}^{\times\, n}, \qquad  [p] \mapsto \Big(p(0), (\d^\ell p)(0), (\d^\ell p)'(0), \ldots, (\d^\ell p)^{(n-1)}(0)\Big)
\]
induces the left trivialization of $T^nG$, \cite{Vizm13}.

Following \cite[Rk. 2.8]{Vizm13}, the (left trivialized) group structure of $T^nG$ may then be presented as
\begin{equation} \label{GrTnG}
(x,\xi_1,\xi_2, \ldots, \xi_n) (y,\zeta_1,\zeta_2,\ldots, \zeta_n) = (xy, \rho_1, \rho_2, \ldots, \rho_n),
\end{equation}
where
\[
\rho_k:= \zeta_k + \sum_{i_1+\ldots+i_\ell = k}\,(-1)^{\ell-1}\,N_{(i_1,\ldots,i_\ell)}\ad_{\zeta_{i_{\ell-1}}}\ldots \ad_{\zeta_{i_1}}\Ad_{y^{-1}}\xi_{i_\ell},
\]
while
\[
N_{(i_1,\ldots,i_\ell)} := \binom{i_1+\ldots+i_\ell - 1}{i_\ell - 1}\binom{i_1+\ldots+i_{\ell-1} - 1}{i_{\ell-1} - 1} \ldots \binom{i_1+i_2 - 1}{i_2 - 1}
\]
is to be the number of anti-lexicographically ordered partitions of $\{1,2,\ldots,k\}$ with cardinalities $i_1,\ldots,i_\ell$.

The multiplication in $T^2G$ has already been considered in \cite{EsenKudeSutl21}. The multiplication in $T^3G$, on the other hand, is
\begin{align} \label{GrT3G}
\begin{split}
& (x,\xi_1,\xi_2,\xi_3) (y,\zeta_1,\zeta_2,\zeta_3) = \\
& \Big(xy, \zeta_1+\Ad_{y^{-1}}\xi_1, \zeta_2+\Ad_{y^{-1}}\xi_2 -\ad_{\zeta_1}\Ad_{y^{-1}}\xi_1,\\
& \hspace{1cm} \zeta_3 + \Ad_{y^{-1}}\xi_3 - 2\ad_{\zeta_1}\Ad_{y^{-1}}\xi_2 - \ad_{\zeta_2}\Ad_{y^{-1}}\xi_1 + \ad^2_{\zeta_1}\Ad_{y^{-1}}\xi_1\Big),
\end{split}
\end{align}
while the multiplication in $T^4G$ is given by
\begin{align} \label{GrT4G}
\begin{split}
& (x,\xi_1,\xi_2,\xi_3,\xi_4) (y,\zeta_1,\zeta_2,\zeta_3, \zeta_4) = \\
& \Big(xy, \zeta_1+\Ad_{y^{-1}}\xi_1, \zeta_2+\Ad_{y^{-1}}\xi_2 -\ad_{\zeta_1}\Ad_{y^{-1}}\xi_1,\\
& \hspace{1cm} \zeta_3 + \Ad_{y^{-1}}\xi_3 - 2\ad_{\zeta_1}\Ad_{y^{-1}}\xi_2 - \ad_{\zeta_2}\Ad_{y^{-1}}\xi_1 + \ad^2_{\zeta_1}\Ad_{y^{-1}}\xi_1, \\
& \zeta_4 +\Ad_{y^{-1}}\xi_4 -  3\ad_{\zeta_1}\Ad_{y^{-1}}\xi_3  - 3\ad_{\zeta_2}\Ad_{y^{-1}}\xi_2 + 3 \ad^2_{\zeta_1}\Ad_{y^{-1}}\xi_2 + \\
& \hspace{1cm} 2 \ad_{\zeta_2}\ad_{\zeta_1}\Ad_{y^{-1}}\xi_1 + \ad_{\zeta_1}\ad_{\zeta_2}\Ad_{y^{-1}}\xi_1 - \ad_{\zeta_3}\Ad_{y^{-1}}\xi_1 - \ad^3_{\zeta_1}\Ad_{y^{-1}}\xi_1\Big).
\end{split}
\end{align}
Accordingly, the unit element is $(e,0,\ldots,0) \in G \times  \G{g}^{\times\,n}$, and 
\[
(x,\xi_1,\ldots,\xi_n)^{-1} = (x^{-1}, \omega_1, \ldots, \omega_n),
\]
where
\[
\omega_k:= - \sum_{i_1+\ldots+i_\ell = k}\,N_{(i_1,\ldots,i_\ell)}\,\Ad_{x}\ad_{\xi_{i_1}}\ldots \ad_{\xi_{i_{\ell-1}}}\xi_{i_\ell}.
\]
In order to be able to distinguish the identical copies of the Lie algebra $\G{g}$, we shall make use of the notation $G \times \G{g}^{(1)} \times  \ldots \times \G{g}^{(n)}$.

\subsection{$n$th iterated tangent group of a Lie group}~

Let now 
\[
T^{(n)}G := \underset{n-many}{\underbrace{T(T(T\ldots (T}}G)))
\]
denotes the iterated tangent bundle of the group $G$. It, then, follows sat once from the left trivialization $tr^L_{TG}:TG\to G\ltimes \G{g}$ that we have the (iterated) left trivialization
\[
tr^L_{T^{(n)}G}: T^{(n)}G \to G \times \G{g}^{\times\,2^n-1} := G \times \G{g}^{(1)} \times  \ldots \times \G{g}^{(2^n-1)}.
\]
The group structure of the (left trivialized) $T^{(n)}G$ was also given in \cite{Vizm13}. Adopting a similar notation therein, let $I_n$ be the set of all anti-lexicographically ordered subsets of $\{1,2,\ldots,n\}$, and let $I_n^\ast := I_n - \{\emptyset\}$. Let also a sequence $(X_\a)_{\a \in I_n^\ast}$ in $\G{g}$ be ordered lexicographically with respect to indexes. Accordingly, 
\[
T^{(n)}G = \{(x,(X_\a)_{\a \in I_n^\ast}) \mid x\in G,\,\,X_\a \in \G{g}\}.
\]
Labeling the identical copies of the Lie algebra $\G{g}$ by the lexicographically ordered anti-lexicographical subsets of $\{1,2,\ldots,n\}$, in lower dimensions we have
\begin{align*}
& T^{(1)}G = TG = G \times \G{g}^{(1)},\\
& T^{(2)}G = G \times \G{g}^{(1)} \times \G{g}^{(2)} \times \G{g}^{(21)}, \\
& T^{(3)}G = G \times \G{g}^{(1)} \times \G{g}^{(2)} \times \G{g}^{(21)} \times \G{g}^{(3)} \times \G{g}^{(31)} \times \G{g}^{(32)} \times \G{g}^{(321)}, \\
& T^{(4)}G = G \times \G{g}^{(1)} \times \G{g}^{(2)} \times \G{g}^{(21)} \times \G{g}^{(3)} \times \G{g}^{(31)} \times \G{g}^{(32)} \times \G{g}^{(321)} \times \\
& \hspace{6cm} \G{g}^{(4)} \times \G{g}^{(41)} \times \G{g}^{(42)} \times \G{g}^{(421)} \times \G{g}^{(43)} \times \G{g}^{(431)} \times \G{g}^{(432)} \times \G{g}^{(4321)}.
\end{align*}
The group operation of $T^{(n)}G$, on the other hand, is given by
\[
(x,(X_\a)_{\a \in I_n^\ast}) (y,(Y_\a)_{\a \in I_n^\ast}) = (xy,(Z_\a)_{\a \in I_n^\ast}),
\]
where $\C{P}(\a)$ being the set of all anti-lexicographically ordered partitions of $\a \in I_n^\ast$, and considering a partition $\lambda \in \C{P}(\a)$ of length $1\leq \ell \leq n$ as $\a = \lambda_\ell \cup \ldots \cup\lambda_1$, 
\[
Z_\a = Y_\a +\sum_{\lambda\in \C{P}(\a)}\,(-1)^{\ell-1}\,\ad_{Y_{\lambda_{\ell-1}}}\ldots \ad_{Y_{\lambda_1}}\Ad_{y^{-1}}X_{\lambda_\ell}.
\]
Furthermore, the inversion may be given as
\[
(x,(X_\a)_{\a \in I_n^\ast})^{-1} = (x,(W_\a)_{\a \in I_n^\ast}),
\]
where
\[
W_\a = -\sum_{\lambda\in \C{P}(\a)}\,\Ad_x\ad_{X_{\lambda_1}} \ldots \ad_{X_{\lambda_{\ell-1}}} X_{\lambda_\ell}.
\]

The following proposition sheds further light on the group structure of the iterated tangent group, which was investigated (from a similar point of view) in \cite[Prop. 3.4]{EsenKudeSutl21}. It was shown therein that $T^{(2)}G \cong G \times \G{g}^{(1)}\times \G{g}^{(2)} \times \G{g}^{(21)}$ is the double cross product of two subgroups $T^2G \cong G \times \G{g}^{(2)} \times \G{g}^{(21)}$, and $\G{g}\cong\G{g}^{(1)}$. It turns out that for $n\geq 3$ this is no longer the case; the double cross product structure evolves into a cocycle double cross product, which is the subject of the following proposition.

\begin{proposition}\label{prop-TnG-bicocycle-decomp}
The $n$th order iterated tangent group $T^{(n)}G$ has the structure of a cocycle double cross product of the $n$th order tangent group $T^nG$, by $\G{g}^{\times \,2^n-1-n}$. In short, 
\begin{equation}\label{TnG-cocycle-ext}
T^{(n)}G \cong  \G{g}^{\times \,2^n-1-n} \bowtie_\g T^nG
\end{equation}
for some $\g:\G{g}^{\times \,2^n-1-n} \times \G{g}^{\times \,2^n-1-n} \to T^nG$.
\end{proposition}

\begin{proof}
We embed $T^nG$ into $T^{(n)}G$ via
\[
T^nG \to T^{(n)}G, \qquad (x,(\xi_k)_{1\leq k \leq n}) \mapsto (x,(X_\a)_{\a \in I_n^\ast})
\]
where 
\[
X_\a := \begin{cases}
\xi_1 & \text{ if  } |\a| = 1, \\
\xi_2 & \text{ if  } |\a| = 2, \\
\vdots & \vdots \\
\xi_n & \text{ if  } |\a| = n. 
\end{cases}
\]
It follows at once from the group structure of $T^{(n)}G$ that $T^nG \subseteq T^{(n)}G$ is a subgroup. However its complement, that is the subset of the elements of the form $(e,(X_\a)_{\a \in I_n^\ast}) \in T^{(n)}G$ where 
\[
X_\a = 0 \text{  if  } \a \in \{n, n(n-1),\ldots, n(n-1)\ldots 1\},
\]
(which is isomorphic to $\G{g}^{\times \,2^n-1-n}$) is not a subgroup; it is not closed under the multiplication on $T^{(n)}G$. Furthermore, it follows from a straightforward computation that, along with the above presentations of the elements, the multiplication on $T^{(n)}G$ yields an isomorphism 
\[
\G{g}^{\times \,2^n-1-n} \times T^nG \to T^{(n)}G.
\] 
Accordingly, the claim becomes a direct consequence of \cite[Prop. 2.2]{EsenGuhaSutl22}.
\end{proof}

For the sake of the convenience, let us record below the further details of $T^{(3)}G \cong \G{g}^{\times 4} \bowtie_\gamma T^3G$.

\begin{example}
Given $(x,X_1,X_2,X_{21},X_3,X_{31},X_{32},X_{321}),(y,Y_1,Y_2,Y_{21},Y_3,Y_{31},Y_{32},Y_{321}) \in T^{(3)}G$, we recall from \cite{Vizm13} that the multiplication on $T^{(3)}G$ may be given by
\begin{align*}
& (x,X_1,X_2,X_{21},X_3,X_{31},X_{32},X_{321}) (y,Y_1,Y_2,Y_{21},Y_3,Y_{31},Y_{32},Y_{321}) = \\
& (xy, Y_1 + \Ad_{y^{-1}}(X_1), Y_2 + \Ad_{y^{-1}}(X_2), Y_{21} + \Ad_{y^{-1}}(X_{21}) - \ad_{Y_1}\Ad_{y^{-1}}(X_2), Y_3 + \Ad_{y^{-1}}(X_3), \\
& \hspace{2cm} Y_{31} + \Ad_{y^{-1}}(X_{31}) - \ad_{Y_1}\Ad_{y^{-1}}(X_3), Y_{32} + \Ad_{y^{-1}}(X_{32}) - \ad_{Y_2}\Ad_{y^{-1}}(X_3), \\
& Y_{321} + \Ad_{y^{-1}}(X_{321}) - \ad_{Y_1}\Ad_{y^{-1}}(X_{32}) - \ad_{Y_2}\Ad_{y^{-1}}(X_{31}) + \ad_{Y_2}\ad_{Y_1}\Ad_{y^{-1}}(X_3)).
\end{align*}
On the other hand, the embeddings $\G{g}^{\times 4} \hookrightarrow T^{(3)}G \hookleftarrow T^3G$ of Proposition \ref{prop-TnG-bicocycle-decomp} are given by 
\[
T^3G\ni (x,\xi_1,\xi_2,\xi_3) \mapsto (x,\xi_1,\xi_1,\xi_2,\xi_1,\xi_2,\xi_2,\xi_3) \in T^{(3)}G,
\]
and
\[
\G{g}^{\times 4}\ni (X_1,X_2,X_{21},X_{31}) \mapsto (X_1,X_2,X_{21},0,X_{31},0,0) \in T^{(3)}G.
\]
Accordingly, the twisted cocycle map and the induced multiplication turn out to be
\begin{align*}
&\g: \G{g}^{\times 4} \times \G{g}^{\times 4} \to T^3G, \\
& \Big((e,X_2,X_3,X_{31},X_{32}),(e,Y_2,Y_3,Y_{31},Y_{32})\Big) \mapsto (e,0,0,-\ad_{Y_2}X_{31}), \\
&\Phi: \G{m} \times \G{m} \to \G{m} , \\
& (e,X_1,X_2,X_{21},X_{31})\cdot (e,Y_1,Y_2,Y_{21},Y_{31}) = (e,Y_1+X_1,Y_2+X_2,Y_{21}+X_{21}-\ad_{Y_1}X_2,Y_{31}+X_{31}).
\end{align*}

In order to obtain $\s: T^3G \times \G{g}^{\times 4} \to T^3G$ and $\rt: T^3G \times \G{g}^{\times 4} \to \G{g}^{\times 4}$ we compute, in $T^{(3)}G$,
\begin{align*}
& (x,\xi_1,\xi_1,\xi_2,\xi_1,\xi_2,\xi_2,\xi_3) (e,X_1,X_2,X_{21},0,X_{31},0,0)= \\
& (x,X_1+\xi_1,X_2+\xi_1, X_{21}+\xi_2 - \ad_{X_1}\xi_1,\xi_2, \\
& \hspace{2cm} X_{31}+\xi_2-\ad_{X_1}\xi_1,\xi_2-\ad_{X_2}\xi_1, \xi_3-\ad_{X_1}\xi_2-\ad_{X_2}\xi_2 +\ad_{X_2}\ad_{X_1}\xi_1),
\end{align*}
and compare it with
\begin{align*}
& (e,Y_1,Y_2,Y_{21},0,Y_{31},0,0)(y,\eta_1,\eta_1,\eta_2,\eta_2,\eta_1,\eta_2,\eta_2,\eta_3) = \\
& (y, \eta_1+\Ad_{y^{-1}}Y_1,\eta_1+\Ad_{y^{-1}}Y_2,\eta_2 +\Ad_{y^{-1}}Y_{21}-\ad_{\eta_1}\Ad_{y^{-1}}Y_2, \\
& \hspace{2cm}\eta_1, \eta_2+\Ad_{y^{-1}}Y_{31}, \eta_2, \eta_3 - \ad_{\eta_1}\Ad_{y^{-1}}Y_{31}).
\end{align*}
We thus find
\begin{align*}
& (x,\xi_1,\xi_1,\xi_2,\xi_1,\xi_2,\xi_2,\xi_3) (e,X_1,X_2,X_{21},0,X_{31},0,0)= \\
& (e, \Ad_x(X_1+\xi_1-\xi_2),\Ad_x (X_2+\xi_1-\xi_2) , \\
& \hspace{2cm} \Ad_x (X_{21}+\ad_{X_2-X_1}\xi_1-\ad_{\xi_2}(X_2+\xi_1)), 0, \Ad_x (X_{31}+\ad_{X_2-X_1}\xi_1) ,0,0) \times \\
& (x,\xi_2,\xi_2,\xi_2-\ad_{X_2}\xi_1,\xi_2,\xi_2-\ad_{X_2}\xi_1, \\
& \hspace{2cm}\xi_2-\ad_{X_2}\xi_1, \xi_3-\ad_{X_2+X_1}\xi_2+\ad_{\xi_2}(X_{31}+\ad_{X_2-X_1}\xi_1)+\ad_{X_2}\ad_{X_1}\xi_1),
\end{align*}
that is,
\begin{align*}
&\s: T^3G \times \G{g}^{\times 4} \to T^3G , \\
& \Big((x,\xi_1,\xi_2,\xi_3), (e,X_1,X_2,X_{21},X_{31})\Big) \mapsto  (x,\xi_2,\xi_2,\xi_2-\ad_{X_2}\xi_1,\xi_2,\xi_2-\ad_{X_2}\xi_1, \\
& \hspace{4cm}\xi_2-\ad_{X_2}\xi_1, \xi_3-\ad_{X_2+X_1}\xi_2+\ad_{\xi_2}(X_{31}+\ad_{X_2-X_1}\xi_1)+\ad_{X_2}\ad_{X_1}\xi_1), \\
&\rt: T^3G \times \G{g}^{\times 4} \to \G{g}^{\times 4} , \\
& (x,\xi_1,\xi_2,\xi_3)\rt (e,X_1,X_2,X_{21},X_{31})=  (e, \Ad_x(X_1+\xi_1-\xi_2),\Ad_x (X_2+\xi_1-\xi_2) , \\
& \hspace{4cm} \Ad_x (X_{21}+\ad_{X_2-X_1}\xi_1-\ad_{\xi_2}(X_2+\xi_1)), 0, \Ad_x (X_{31}+\ad_{X_2-X_1}\xi_1) ,0,0).
\end{align*}
\end{example}

\begin{remark}
As a result of Proposition \ref{prop-TnG-bicocycle-decomp}, $\G{g}^{\times \,2^n-1-n} \subseteq T^{(n)}G$ is not a subgroup for $n\geq 3$. Therefore, the dynamics on $T^nG \subseteq T^{(n)}G$ cannot be obtained from that on $T^{(n)}G$ by a $\G{g}^{\times \,2^n-1-n}$-reduction, as was done in \cite{EsenKudeSutl21} for $n=2$. However, the immersion $T^nG \to T(T^{n-1}G)$ of \cite[Subsect. 2.3]{colombo2011geometry} offers an alternate route to walk around this problem.
\end{remark}

\begin{proposition}
The $n$th order tangent group $T^nG$ may be seen as a subgroup of the tangent group $T(T^{n-1}G)$ via
\begin{equation}\label{TnG-into-T(Tn-1G)}
T^nG \hookrightarrow T(T^{n-1}G), \qquad (x,\xi_1,\ldots,\xi_n) \mapsto (x,\xi_1,\ldots,\xi_{n-1}; \xi_1,\ldots,\xi_n).
\end{equation}
\end{proposition}

\begin{proof}
It suffices to show that \eqref{TnG-into-T(Tn-1G)} respects the group multiplications. To this end, we shall use induction over $k\geq 0$. The case $k=1$ has already been treated in \cite{EsenKudeSutl21}. Let us consider $((x,\tilde{\xi}),\xi_{k+1}):=(x,\xi_1,\ldots,\xi_{k+1})$ and $((y,\tilde{\eta}),\eta_{k+1}):=(y,\eta_1,\ldots,\eta_{k+1})$ in $T^{k+1}G$. Then, \eqref{TnG-into-T(Tn-1G)} operates as
\[
T^{k+1}G \ni ((x,\tilde{\xi}),\xi_{k+1}) \mapsto (x,\tilde{\xi}; \tilde{\xi},\xi_{k+1}) \in T(T^kG)
\]
and
\[
T^{k+1}G \ni ((y,\tilde{\eta}),\eta_{k+1}) \mapsto (y,\tilde{\eta}; \tilde{\eta},\eta_{k+1}) \in T(T^kG).
\]
Since $T^{k+1}G$ is an abelian extension of $T^kG$, see for instance \cite[Prop. 4.1]{Vizm13},
\begin{align*}
& (x,\xi_1,\ldots,\xi_{k+1})(y,\eta_1,\ldots,\eta_{k+1}) = ((x,\tilde{\xi}),\xi_{k+1})((y,\tilde{\eta}),\eta_{k+1}) = \\
&  ((x,\tilde{\xi})(y,\tilde{\eta}), \eta_{k+1} + \Ad_{y^{-1}}\xi_{k+1} +\g((x,\tilde{\xi}),(y,\tilde{\eta})))
\end{align*}
where $\g:T^kG\times T^kG \to \G{g}$ is the group 2-cocycle characterizing the abelian extension (the explicit expression of which may be found in \cite[Sect. 4]{Vizm13}).
On the other hand,
\begin{align*}
& (x,\tilde{\xi}; \tilde{\xi},\xi_{k+1})(y,\tilde{\eta}; \tilde{\eta},\eta_{k+1}) = ((x,\tilde{\xi})(y,\tilde{\eta});\Ad_{(y,\tilde{\eta})^{-1}}(\tilde{\xi},\xi_{k+1}) + (\tilde{\eta},\eta_{k+1})) = \\
& ((x,\tilde{\xi})(y,\tilde{\eta});\Ad_{(y,\tilde{\eta})^{-1}}(\tilde{\xi},0) + \Ad_{(y,\tilde{\eta})^{-1}}(\t,\xi_{k+1}) + (\tilde{\eta},\eta_{k+1})), 
\end{align*}
where it follows at once from Proposition \ref{prop-Ad-cocycle-cross} above\footnote[1]{To be more precise, it follows from the analogue result for $H {}_{\g\hspace{-.1cm}}\bowtie M$.} that 
\[
\Ad_{(y,\tilde{\eta})^{-1}}(\t,\xi_{k+1}) = (\t,\Ad_{y^{-1}}\xi_{k+1}).
\]
Let, next, $(xy, \tilde{\eta} + \tilde{\zeta} ):=(x,\tilde{\xi})(y,\tilde{\eta}) \in T^kG$. Then, $T^{k+1}G$ being an abelian extension of $T^kG$, it follows from the induction step on $T^kG$ that
\[
\Ad_{(y,\tilde{\eta})^{-1}}(\tilde{\xi},0) = (\tilde{\zeta}, \g((x,\tilde{\xi}),(y,\tilde{\eta}))).
\]
We remark that the cocycle term on the latter component does not depent on $x\in G$, see \cite[Prop. 4.1]{Vizm13}, as such the equality makes sense. The result then follows.
\end{proof}

The following result is straightforward, hence the proof is omitted.

\begin{proposition}
The subgroup $\G{g}^{\times \,n-1} \cong T^{n-1}_eG := \{(e,\xi_1,\ldots,\xi_{n-1}) \in T^{n-1}G \mid \xi_1,\ldots,\xi_{n-1}\in \G{g}\}$ of $T^{n-1}G$, may be seen as a subgroup of the tangent group $T(T^{n-1}G)$ via
\begin{equation}\label{gn-1-into-T(Tn-1G)}
\G{g}^{\times \,n-1} \hookrightarrow T(T^{n-1}G), \qquad (X_1,\ldots,X_{n-1}) \mapsto (e,X_1,\ldots, X_{n-1}; 0,\ldots,0).
\end{equation}
\end{proposition}

\begin{proposition}\label{prop-higher-order-TG-factors}
The tangent group $T(T^{n-1}G)$ has the structure of a double cross product of the $n$th order tangent group $T^nG$, by $\G{g}^{\times \,n-1}$. In short, 
\begin{equation}\label{T(TnG)-double-cross}
T(T^{n-1}G) \cong  \G{g}^{\times \,n-1} \bowtie T^nG.
\end{equation}
\end{proposition}

\begin{proof}
In view of the embeddings \eqref{TnG-into-T(Tn-1G)} and \eqref{gn-1-into-T(Tn-1G)}, we only need to observe that the group multiplication on $T(T^{n-1}G)$ yields an isomorphism
\[
\G{g}^{\times \,n-1}\times T^nG \to T(T^{n-1}G).
\]
Indeed, we have
\begin{align*}
& (e,X_1,\ldots, X_{n-1}; 0,\ldots,0)(x,\xi_1,\ldots,\xi_{n-1}; \xi_1,\ldots,\xi_n)= \\
& ((e,X_1,\ldots, X_{n-1})(x,\xi_1,\ldots,\xi_{n-1}); \xi_1,\ldots,\xi_n),
\end{align*}
where the multiplication on the former components takes place in $T^{n-1}G$, and given any $(y,\eta_1,\ldots,\eta_{n-1}) \in T^{n-1}G$, the indeterminants $x, X_1,\ldots, X_{n-1}$ may be arranged in such a way that 
\[
(e,X_1,\ldots, X_{n-1})(x,\xi_1,\ldots,\xi_{n-1}) = (y,\eta_1,\ldots,\eta_{n-1}).
\]
Indeed, one would set $x=y$, and 
\[
(e,X_1,\ldots, X_{n-1}) = (y,\eta_1,\ldots,\eta_{n-1})(y,\xi_1,\ldots,\xi_{n-1})^{-1} = (y,\eta_1,\ldots,\eta_{n-1})(y^{-1},\zeta_1,\ldots,\zeta_{n-1})
\]
for some $\zeta_1,\ldots,\zeta_{n-1} \in \G{g}$. The result thus follows.
\end{proof}

Let us illustrate the mutual actions in $T(T^2G) \cong  \G{g}^{\times \,2} \bowtie T^3G$ below.

\begin{example}
In view of the inclusions
\begin{align}
&T^3G \to T(T^2G),\quad (x,\xi_1,\xi_2,\xi_3)\to (X,\xi_1,\xi_2, \xi_1, \xi_2, \xi_3),\\
&\G{g}^{\times \,2} \to T(T^2G),\quad (X_1,X_2)\to (e,X_1,X_2,0,0,0),
\end{align}
the multiplication 
\begin{align*}
 \G{g}^{\times \,2} \times T^3 G&\to T(T^2G)\\
\left(({X}_1,{X}_2);(X,\xi_1,\xi_2,\xi_3)\right)&\mapsto(e,X_1,X_{2},0,0,0)(X,\xi_1,\xi_2,\xi_1,\xi_2,\xi_3
)=\\
&((e,X_1,X_{2})(X,\xi_1,\xi_2),(\xi_1,\xi_2,\xi_3))=\\
&(X,\xi_1+\Ad_{X^{-1}}X_1,\xi_2+\Ad_{X^{-1}}X_2-\ad_{\xi_1}\Ad_{X^{-1}}X_1,\xi_1,\xi_2,\xi_3)
\end{align*}  
happens, clearly, to be an isomorphism. 
 
On the other hand, in order to compute the mutual actions of $\G{g}^{\times \,2}$ and $T^3G$ onto each other we compare
\begin{align*}
&(X,\xi_1,\xi_2,\xi_1,\xi_2,\xi_3)(e,X_1,X_{2},0,0,0)=((X,\xi_1,\xi_2)(e,X_1,X_{2}),\Ad_{(e,X_1,X_{2})^{-1}}(X,\xi_1,\xi_2))\\
&=(X, X_1+\xi_1,X_2+\xi_2-\ad_{X_1}\xi_1,\xi_1,\xi_2+\ad_{\xi_1}(X_1),\xi_3+\ad_{\xi_1}(X_2)+2\ad_{\xi_2}X_1-\ad_{X_1}\ad_{\xi_1}X_1),
\end{align*}  
where the adjoint action is
\begin{align*}
&\Ad_{(e,-X_1,-X_2)}(\xi_1,\xi_2,\xi_3)=\left.\frac{d}{dt}\right|_{t=0}(e,-X_1,-X_2)(X_t,\xi_2t,\xi_3t)(e,X_1,X_2) =\\
&\left.\frac{d}{dt}\right|_{t=0}(X_t,\xi_2t+\Ad_{X_t^{-1}}(-X_1),\xi_3t+\Ad_{X_t^{-1}}(-X_2)-\ad_{\xi_2 t}\Ad_{X_t^{-1}}(-X_1))(e,X_1,X_2) =\\
&\left.\frac{d}{dt}\right|_{t=0}(X_t,X_1+\xi_2t+\Ad_{X_t^{-1}}(-X_1),X_2+\xi_3t+\\
&\hspace{4cm}\Ad_{X_t^{-1}}(-X_2)-\ad_{\xi_2 t}\Ad_{X_t^{-1}}(-X_1)-\ad_{X_1}(\xi_2t+\Ad_{X_t^{-1}}(-X_1))) =\\
&(\xi_1,\xi_2-\ad_{\xi_1}(-X_1),\xi_3-\ad_{\xi_1}(-X_2)-\ad_{\xi_2}(-X_1)-\ad_{X_1}(\xi_2-X_1)-\ad_{X_1}(-\ad_{\xi_1}(-X_1)),
\end{align*} 
with
\begin{align*}
&(e,Y_1,Y_{2},0,0,0)(Y,\eta_1,\eta_2,\eta_1,\eta_2,\eta_3)=((e,Y_1,Y_{2})(Y,\eta_1,\eta_2),\Ad_{(Y,\eta_1,\eta_2)^{-1}}(e,Y_1,Y_{2}))\\
&(Y,\eta_1+\Ad_{Y^{-1}}Y_1,\eta_2+\Ad_{Y^{-1}}Y_2-\ad_{\eta_1}\Ad_{Y^{-1}}Y_1,\eta_1,\eta_2,\eta_3).
\end{align*}
Hence, setting
\begin{align*}
&Y=X, \quad Y_1=\Ad_Y X_1,\quad Y_2=\Ad_X(X_2+\ad_{\eta_1}\Ad_{X^{-1}}Y_1)\\
&\eta_1=\xi_1,\quad  \eta_2=\xi_2+\ad_{\xi_1}X_1,\quad
\eta_3=\xi_3+\ad_{\xi_1}(X_2)+2\ad_{\xi_2}X_1-\ad_{X_1}\ad_{\xi_1}X_1,
\end{align*}
the mutual actions appear as
\begin{align*}
& \rt:T^3G\times \G{g}^{\times \,2}\to \G{g}^{\times \,2},\\
& (X,\xi_1,\xi_2, \xi_3)\rt(X_1,X_2)=(\Ad_X X_1,\Ad_X(X_2+\ad_{\xi_1} \Ad_{X^{-1}}\Ad_yX_1)), 
\end{align*}
and
\begin{align*}
&\lt:T^3G\times \G{g}^{\times \,2}\to T^3G, \\
& (X,\xi_1,\xi_2, \xi_3)\lt(X_1,X_2)=(X,\xi_1,\xi_2+\ad_{\xi_1} X_1,\xi_3+\ad_{\xi_1}X_2+2\ad_{\xi_2}X_1-\ad_{X_1}\ad_{\xi_1}X_1). 
 \end{align*}
\end{example}  

We shall next obtain the Euler-Lagrange equations over $T(T^2G)$. Accordingly, in view of Proposition \ref{prop-higher-order-TG-factors}, the Euler-Lagrange equations over $T^3G$ will follow from the $\G{g}^{\times \, 2}$-reduction.

On the other hand, let us recall from \cite[Prop. 3.1]{EsenKudeSutl21} that the second order tangent group $T^2G$ is a 2-cocycle extension of $TG$ by $\G{g}$, that is,
\[
T^2G \cong TG \ltimes_\vp \G{g},
\]
where
\[
\vp: TG \times TG \to \G{g}, \qquad \vp((g_1,\xi_1),(g_2,\xi_2)) := -\ad_{\xi_2}\Ad_{g_2^{-1}}\xi_1.
\]
Hence, in view of \eqref{derv-of-gamma}, the Lie algebra of $T^2G \cong TG \ltimes_\vp \G{g}$ is of the form of a 2-cocycle extension $(\G{g}\ltimes \G{g}) \ltimes_\t \G{g}$, where
\[
\t: (\G{g}\ltimes \G{g}) \ot (\G{g}\ltimes \G{g}) \to (\G{g}\ltimes \G{g}), \qquad \t((\xi_1,\xi_2),(\xi_3,\xi_4)) := -2\ad_{\xi_4}\xi_2.
\]

\begin{proposition}
The Euler-Lagrange equations of a  Lagrangian $\C{L}=\C{L}(g,\xi_1,\xi_2,\eta_0,\eta_1,\eta_2)$
defined on $T(T^2G) \cong(TG\ltimes_{\varphi} \G{g}) \ltimes((\G{g} \ltimes\G{g})\ltimes_\t \G{g})$ are given by
\begin{align}
&\frac{d}{dt}\frac{\delta\C{L}}{\delta \eta_0}=T^\ast L_{g}(\frac{\delta\C{L}}{\delta g})-\ad^{\ast}_{\xi_1}(\frac{\delta\C{L}}{\delta \xi_1})-\ad^{\ast}_{\xi_2}(\frac{\delta\C{L}}{\delta \xi_2})-\ad^{\ast}_{\eta_0}(\frac{\delta\C{L}}{\delta \eta_0})-\ad^{\ast}_{\eta_1}(\frac{\delta\C{L}}{\delta \eta_1})-\ad^{\ast}_{\eta_2}(\frac{\delta\C{L}}{\delta \eta_2})\label{ELET3-1}\\
&\frac{d}{dt}\frac{\delta\C{L}}{\delta \eta_1}=\frac{\delta\C{L}}{\delta \xi_1}-\ad^{\ast}_{\xi_1}(\frac{\delta\C{L}}{\delta \xi_2})-\ad^{\ast}_{\eta_0}(\frac{\delta\C{L}}{\delta \eta_1})-2\ad^{\ast}_{\eta_1}(\frac{\delta\C{L}}{\delta \eta_2})\label{ELET3-2}\\
&\frac{d}{dt}\frac{\delta\C{L}}{\delta \eta_2}=T^{\ast}L_{\xi_2}\left(\frac{\delta\C{L}}{\delta \xi_2}\right)-\ad^{\ast}_{\eta_0}(\frac{\delta\C{L}}{\delta\eta_2}).\label{ELET3-3}
\end{align} 
\end{proposition}

\begin{proof}
The 2-cocycle extension nature of $T^2G$ forces the Euler-Lagrange equations to be given by \eqref{mEL} and \eqref{mEL-II}. Keeping in mind that the left action of $\G{g}^{}\times\,2$ on $TG$ is trivial, we have
 \begin{align}
 &\frac{d}{dt}\left(\frac{\delta\C{L}}{\delta \eta_0},\frac{\delta\C{L}}{\delta \eta_1}\right)   = T^\ast L_{(g,\xi_1)}(\frac{\delta\C{L}}{\delta g},\frac{\delta\C{L}}{\delta \xi_1})+ T^\ast\varphi_{(g,\xi_1)}(\frac{\delta\C{L}}{\delta \xi_2})+ T^\ast\sigma_{\xi_2}(\frac{\delta\C{L}}{\delta \xi_2})\label{EPET3}\\
  &\hspace{3cm}-\G{ad}^\ast_{(\eta_0,\eta_1)}\left(\frac{\delta\C{L}}{\delta \eta_0},\frac{\delta\C{L}}{\delta \eta_1}\right)+\G{a}_{\eta_2}^{\ast}\left(  \frac{\delta\C{L}}{\delta \eta_2}\right)+\theta^{\ast}_{(\eta_0,\eta_1)}\left(  \frac{\delta\C{L}}{\delta \eta_2}\right)\notag,\\
&\frac{d}{dt}\frac{\delta\C{L}}{\delta\eta_2}   =T^\ast L_{\xi_2}(\frac{\delta\C{L}}{\delta \xi_2})-\ad_{\eta_2}^{\ast}\left(\frac{\delta\C{L}}{\delta\eta_2}\right)-\mathfrak{b}_{(\eta_0,\eta_1)}^{\ast}\left(\frac
{\delta\C{L}}{\delta \eta_0},\frac{\delta\C{L}}{\delta \eta_1}\right) - 
{}_{(\eta_0,\eta_1)}\psi^{\ast}\left(\frac{\delta\C{L}}{\delta \xi_2}\right).\label{EPET3-II}
\end{align}
The claim then follows by substituting the dual mappings 
\begin{align*}
 &\G{ad}_{(\eta_0,\eta_1)}^{\ast}: \G{g}^{\ast} \oplus \G{g}^{\ast}\to \G{g}^{\ast} \oplus \G{g}^{\ast}, \qquad(\mu_0,\mu_1)\mapsto (\ad^{\ast}_{\eta_0}\mu_0+\ad^{\ast}_{\eta_1}\mu_1,\ad^{\ast}_{\eta_0}\mu_1),\\
 & \G{a}_{\eta_2}^{\ast}: \G{g}^{\ast}\to\G{g}^{\ast} \oplus \G{g}^{\ast}, \qquad \mu_2\mapsto (-\ad^\ast_{\eta_2}\mu_2,0), \\
&\mathfrak{b}_{v}^{\ast}:  \G{g}^{\ast} \oplus \G{g}^{\ast}\to  \G{g}^{\ast}, \qquad (\mu_0,\mu_1)\mapsto \mathfrak{b}_{v}^{\ast}(\mu_0,\mu_1)=0, \\
 & {}_{(\eta_0,\eta_1)}\psi^{\ast}:\G{g}^{\ast}\to \G{g}^{\ast}, \qquad \mu_2\mapsto {}_{(\eta_0,\eta_1)}\psi^{\ast}(\mu_2)=\ad_{\eta_0}^{\ast}(\mu_2), \\
 & \t_{(\eta_0,\eta_1)}^{\ast}:\G{g}^{\ast}\to \G{g}^{\ast} \oplus \G{g}^{\ast}, \qquad \mu_2 \mapsto \t_{(\eta_0,\eta_1)}^{\ast}(\mu_2)=(0,-2\ad^{\ast}_{\eta_1}\mu_2),
 \end{align*}
 and the cotangent lifts 
 \begin{align*}
 &T^{\ast}\varphi_{(g,\xi_1)}:\G{g}^{\ast}\to \G{g}^{\ast} \oplus \G{g}^{\ast},\qquad \mu_2\mapsto T^{\ast}\varphi_{(g,\xi_1)}(\mu_2)=(0,-\ad^{\ast}_{\xi_1}(\mu_2)),\\
 &T^{\ast}\sigma_{\xi_2}:\G{g}^{\ast}\to \G{g}^{\ast} \oplus \G{g}^{\ast},\qquad \mu_2\mapsto T^{\ast}\sigma_{\xi_2}(\mu_2)=(-\ad^{\ast}_{\xi_2}(\mu_2),0),\\
 &T^{\ast}L_{(g,\xi_1)}(\mu_0,\mu_1)=(T^\ast L_g\mu_0-\ad^{\ast}_{\xi_1}(\mu_1), \mu_1).
 \end{align*} 
\end{proof}

The reduction by $T^2G \cong TG \ltimes_\vp \G{g}$ thus yields the Euler-Poincar\'e equations on $(\G{g}\ltimes \G{g}) \ltimes_\t \G{g}$.

\begin{corollary}
 Given a  reduced Lagrangian $\ell=\ell(\nu,\eta,\dot{\eta})=\ell(\eta_0,\eta_1,\eta_2)$ on $(\G{g}\ltimes \G{g}) \ltimes_\t \G{g}$, the Euler-Poincar\'{e} equations turn out to be
\begin{align}
&\frac{d}{dt}\frac{\delta\ell}{\delta \eta_0}=-\ad^{\ast}_{\eta_0}(\frac{\delta\ell}{\delta \eta_0})-\ad^{\ast}_{\eta_1}(\frac{\delta\ell}{\delta \eta_1})-\ad^{\ast}_{\eta_2}(\frac{\delta\ell}{\delta \eta_2}),\label{EPET3-1}\\
&\frac{d}{dt}\frac{\delta\ell}{\delta \eta_1}=-\ad^{\ast}_{\eta_0}(\frac{\delta\ell}{\delta \eta_1})-2\ad^{\ast}_{\eta_1}(\frac{\delta\ell}{\delta \eta_2}),\label{EPET3-2}\\
&\frac{d}{dt}\frac{\delta\ell}{\delta \eta_2}=-\ad^{\ast}_{\eta_0}(\frac{\delta\ell}{\delta\eta_2}).\label{EPET3-3}
\end{align}
\end{corollary}

\begin{remark}
Summing \eqref{EPET3-1}, the first derivative of \eqref{EPET3-2}, and the second derivative of \eqref{EPET3-3}, we obtain 
\begin{equation}
\left(\frac{d}{dt}+\ad^{\ast}_{\eta_0}\right)\left(\frac{\delta\ell}{\delta \eta_0}-\frac{d}{dt}\frac{\delta\ell}{\delta \eta_1}+\frac{d^2}{dt^2}\frac{\delta\ell}{\delta \eta_2}\right)=0
\end{equation}
which coincides with the third order Euler-Poincar\'{e} equation given in \cite[Subsect. 3.1]{gay2012invariant}.
\end{remark}

\section{Illustrations}

\subsection{The Kepler Problem}~

The \emph{energy-momentum space} of the Kepler problem was defined in \cite[Subsect. 3.10]{Marl12} as $\mathbb{R}\times \overrightarrow{\varepsilon} \times \overrightarrow{\varepsilon}$, where $\overrightarrow{\varepsilon}$ stands for the 3-dimensional Euclidean vector space, and $\mathbb{R}$ encodes the energy levels. Explicitly, the bracket on $\overrightarrow{\varepsilon} \times \overrightarrow{\varepsilon}$ that correspond to the energy level $e$ was given by
\begin{equation}
\left[(\vec{v}_1,\vec{\eta}_1),(\vec{v}_2,\vec{\eta}_2)\right]_e =\Big(   \vec{\eta}_1\times \vec{v}_2-\vec{\eta}_2\times \vec{v}_1  ,-\vec{\eta}_1\times \vec{\eta}_2 +\frac{2 e}{m^3k^2}\vec{v}_1\times \vec{v}_2\Big)
\end{equation}
for any $(\vec{v}_1,\vec{\eta}_1),(\vec{v}_2,\vec{\eta}_2)\in \overrightarrow{\varepsilon} \times \overrightarrow{\varepsilon}$. 

The space $\overrightarrow{\varepsilon} \times \overrightarrow{\varepsilon}$ may be considered as a cocycle double cross product of the pair $(\overrightarrow{\varepsilon} ,\overrightarrow{\varepsilon})$ through the twisted cocycle 
\begin{align}
&\t :\overrightarrow{\varepsilon}\otimes \overrightarrow{\varepsilon} \to \overrightarrow{\varepsilon}, \qquad (\vec{v }, \vec{w})\mapsto \t(\vec{v}, \vec{w}) =\frac{2 e}{m^3 k^2} \vec{v}\times\vec{w} ,  \label{theta-map-Kep} 
\end{align}
 and the left action
 \begin{align}
 &\rt: \overrightarrow{\varepsilon}\otimes \overrightarrow{\varepsilon}\to \overrightarrow{\varepsilon}, \qquad (\vec{\eta},\vec{v})\mapsto\vec{\eta} \rt\vec{ v} = \vec{\eta} \times \vec{ v}.\label{left-map-Kep} 
 \end{align} 
The other linear maps are zero, such that
\begin{align}
&\phi :\overrightarrow{\varepsilon}\otimes \overrightarrow{\varepsilon}\to \overrightarrow{\varepsilon}, \qquad \phi(\vec{v},\vec{w}) = 0 \label{phi-map-sym} \\
&\psi: \overrightarrow{\varepsilon}\otimes \overrightarrow{\varepsilon}\to \overrightarrow{\varepsilon}, \qquad \psi(\vec{\eta},\vec{v})=0. \label{right-action-sym}
\end{align}
Accordingly, we have 
\begin{align}
&\mathfrak{b}_v=\overrightarrow{\varepsilon}\to \overrightarrow{\varepsilon}, \qquad \vec{\eta}\mapsto\mathfrak{b}_v(\vec{\eta})=\vec{\eta} \rt \vec{v} = \vec{\eta} \times \vec{ v}, \label{right-action-sym-bv}\\
&\mathfrak{a}_{\eta}= \overrightarrow{\varepsilon}\to \overrightarrow{\varepsilon}, \qquad \vec{v}\mapsto \mathfrak{a}_{\eta}(\vec{v})=\psi(\vec{\eta}, \vec{v})= 0,\label{left-action-sym-an}\\
&{}_{v}\psi: \overrightarrow{\varepsilon}\to \overrightarrow{\varepsilon}, \qquad\vec{\eta}\mapsto {}_{v}\psi(\vec{\eta})=\psi(\vec {\eta},\vec{v})=0,\\
& \t_v: \overrightarrow{\varepsilon}\to \overrightarrow{\varepsilon}, \qquad \vec{w} \mapsto \t_v(\vec{w})=\t(\vec{v},\vec{w}),\\
 &\G{ad}_v:\overrightarrow{\varepsilon}\to \overrightarrow{\varepsilon}, \qquad \vec{w}\mapsto [[\vec{v},\vec{w}]]=0,\\
 &\ad_{\eta}:\overrightarrow{\varepsilon}\to \overrightarrow{\varepsilon}, \qquad \vec{\eta}\mapsto \ad_{\eta}\vec{\beta}=[\vec{\eta},\vec{ \beta}]=-\vec{\eta}\times\vec{\beta}.\label{dual_ad}
\end{align}
Along the way to present the equations of motion, we shall need the dual mappings as well. The left action \eqref{left-map-Kep}, first, dualizes to a right action 
\begin{equation}\label{right-action-sym-dual}
\overset{\ast}{\lt }:\overrightarrow{\varepsilon}^\ast\ot\overrightarrow{\varepsilon} \to \overrightarrow{\varepsilon}^\ast ,\quad (\vec{\mu},\vec{\eta})\mapsto\vec{\mu} \overset{\ast}{\lt} \vec{\eta}=\vec{\mu}\times\vec{\eta},
\end{equation}
while \eqref{right-action-sym} transposes into another trivial mapping
\begin{equation}\label{left-action-sym-dual}
\psi^{\ast}: \overrightarrow{\varepsilon} \ot\overrightarrow{\varepsilon}^\ast\to \overrightarrow{\varepsilon}^\ast  ,\quad (\vec{v},\vec{\alpha})\mapsto\psi^{\ast} (\vec{v},\vec{\alpha})=0,
\end{equation}
just as the dual of \eqref{phi-map-sym} 
\begin{equation}
\phi^{\ast} :\overrightarrow{\varepsilon}\ot \overrightarrow{\varepsilon}^\ast \to \overrightarrow{\varepsilon}^\ast, \qquad (\vec{v},\vec{\mu}) \mapsto\phi^{\ast}(\vec{v},\vec{\mu}) = 0.
\end{equation}
We thus have
\begin{align}
&\mathfrak{b}^{\ast}_v= \overrightarrow{\varepsilon}^\ast\to \overrightarrow{\varepsilon}^{\ast} , \qquad \vec{\mu}\mapsto\mathfrak{b}^{\ast}_v\vec{\mu} = -\vec{v}\times\vec{\mu}, \\
&\mathfrak{a}^{\ast}_{\eta}=  \overrightarrow{\varepsilon}^{\ast} \to \overrightarrow{\varepsilon}^\ast , \qquad \vec{\alpha}\mapsto\mathfrak{a}^{\ast}_{\eta}(\vec{\alpha})= 0,\\
&{}_v\psi^{\ast}: \overrightarrow{\varepsilon}^\ast \to \overrightarrow{\varepsilon}^\ast ,\qquad \vec{\alpha}\mapsto{}_v\psi^{\ast}(\vec{\alpha})=0,\label{left-action-sym-dual}\\
&\theta^\ast_v:\overrightarrow{\varepsilon}^\ast \to \overrightarrow{\varepsilon}^\ast,\qquad \vec{\alpha}\mapsto\theta^\ast_v(\vec{\alpha})=-\frac{2 e}{m^3 k^2} \vec{v}\times \vec{\alpha},\\
 &\G{ad}_v^{\ast}:\overrightarrow{\varepsilon}^\ast\to \overrightarrow{\varepsilon}^\ast, \qquad \vec{\mu}\mapsto \G{ad}_v^{\ast}(\vec{\mu})=0,\\
&\ad^{\ast}_{\eta}:\overrightarrow{\varepsilon}^{\ast}\to \overrightarrow{\varepsilon}^{\ast},\qquad \vec{\alpha}\mapsto \ad^{\ast}_{\eta}(\vec{\alpha})=\vec{\eta}\times \vec{\alpha}.
\end{align}
Substituting these into \eqref{EPE} and \eqref{EPE-II}, we obtain the Euler-Poincar\'{e} equations  
\begin{align}
&\frac{d}{dt}\frac{\delta \ell}{\delta \vec{v}}=\frac{\delta \ell}{\delta \vec{v}}\times \vec{n}+\frac{2 e}{m^3 k^2} \frac{\delta \ell}{\delta \vec{\eta}}\times \vec{v}\\
&\frac{d}{dt}\frac{\delta \ell}{\delta \vec{\eta}}=\frac{\delta \ell}{\delta \vec{\eta}}\times\vec{\eta}-\vec{v}\times  \frac{\delta \ell}{\delta \vec{v} }
\end{align}
associated to a (reduced) Lagrangian $\ell = \ell(\vec{v},\vec{\eta})$ on $\overrightarrow{\varepsilon} \times \overrightarrow{\varepsilon}$.

Similarly, given a (reduced) Hamiltonian $\G{H} = \G{H}(\kappa,\lambda)$ on $\overrightarrow{\varepsilon}^\ast \times \overrightarrow{\varepsilon}^\ast$, the Lie-Poisson equations \eqref{LiePoissonEqn-I} and \eqref{LiePoissonEqn-II} become
\begin{align}
&\frac{d \vec{\kappa}}{dt}=-\vec{\kappa}\times(\frac{\delta \G{H}}{\delta \vec{\lambda}})+ \frac{2 e}{m^3 k^2}\frac{\delta \G{H}}{\delta \vec{\kappa}}\times \vec{\lambda},\\
&\frac{d \vec{\lambda}}{dt}=\frac{\delta \G{H}}{\delta \vec{\lambda}}\times \vec{\lambda}+\frac{\delta \G{H}}{\delta \vec{\kappa}}\times \vec{\kappa}.
\end{align}

\subsection{Nonlinear Tokamak Dynamics}~

As a second illustration, we shall now consider the Hamiltonian four-field model of \cite{HazeHsuMorr87,HazeKotsMorr85}. The underlying Lie algebra of this model has been reconsidered in \cite[Subsect. 2.3.2]{ThifMorr00} from the point of view of Lie algebra extensions of the form $\G{g} \oplus \G{g} \oplus \ldots \oplus \G{g}$, say $n$-copies, with a bracket of the form
\[
[(\xi_1,\ldots,\xi_n),(\eta_1,\ldots,\eta_n)] := \left(\sum_{i,j=1}^n W_1^{ij}[\xi_i,\eta_j], \sum_{i,j=1}^n W_2^{ij}[\xi_i,\eta_j], \ldots, \sum_{i,j=1}^n W_n^{ij}[\xi_i,\eta_j]\right)
\]
for suitable constants $W_k^{ij}$, $1\leq i,j,k\leq n$. Any such Lie algebra is necessarily subsumed in the most general construction introduced in \cite{EsenGuhaSutl22}. In particular, the Lie algebra underlying the four-field model is of the form of a unified product, that is a cocycle double cross product. More precisely, along the lines of \cite{ThifMorr00}, $\G{g}$ being the Lie algebra of the group of volume preserving diffeomorphisms (on a two-dimensional domain), and $\G{g}_1$, $\G{g}_2$, $\G{g}_3$ and $\G{g}_4$ being copies of $\G{g}$, the bracket is given by
\begin{align}\label{Bracket_tokamak}
\begin{split}
&[(v, \beta, w, \alpha) , (v', \beta', w', \alpha')] =\\ 
&\hspace{2cm} \Big([\alpha,v']+ [v, \alpha'],[\alpha,\beta']+ [\beta, \alpha'], [\alpha,w']+[w,\alpha']-B_i([\beta,v']+[v,\beta']),[\alpha,\alpha'] \Big)
\end{split}
\end{align}
for any two $( v,\beta, w,\alpha),( v',\beta', w',\alpha')\in \G{g}_2 \oplus \G{g}_3\oplus  \G{g}_4\oplus \G{g}_1$\footnote{The indexing serves the presentation to fit our convention $\G{m}\rtimes_\t \G{h}$ with $\G{h}$ being the subalgebra.}, where $B_i$ is a parameter that measures the compressibility. We thus have a cocycle double cross sum Lie algebra $(\G{g}_2\oplus\G{g}_3)\bowtie_\t(\G{g}_4\oplus\G{g}_1 ):=(\G{g}_2 \oplus\G{g}_3)\oplus (\G{g}_4\oplus\G{g}_1 )$, where the left action \eqref{left-action}, and the linear maps \eqref{phi-map} - \eqref{right-action} appear as 
\begin{align}
&\rt: (\G{g}_4 \oplus\G{g}_1)\ot (\G{g}_2\oplus\G{g}_3)\to \G{g}_2\oplus\G{g}_3, \quad (w,\alpha) \rt (v,\beta) = (\ad_\alpha v, \ad_\alpha\beta),\label{left-action-Tok}\\
&\phi :(\G{g}_2 \oplus\G{g}_3)\ot(\G{g}_2\oplus\G{g}_3) \to \G{g}_2\oplus\G{g}_3, \quad \phi((v,\beta),(v',\beta')) =(0,0), \label{phi-map-Tok} \\
&\t :(\G{g}_2\oplus\G{g}_3) \ot (\G{g}_2\oplus\G{g}_3) \to \G{g}_4\oplus\G{g}_1, \quad \t((v,\beta),(v',\beta'))=(-B_i(\ad_\beta v'+\ad_v\beta'),0) ,  \label{theta-map-Tok} \\
&\psi: (\G{g}_4\oplus\G{g}_1) \ot(\G{g}_2\oplus\G{g}_3)\to \G{g}_4\oplus\G{g}_1, \quad \psi((w,\alpha) ,( v,\beta))=(0,0). \label{right-action-Tok}
\end{align}
The adjoint action of $\G{g}_4\oplus\G{g}_1$ on itself, on the other hand, takes the form
\begin{align*}
\ad: (\G{g}_4\oplus\G{g}_1)\ot (\G{g}_4\oplus\G{g}_1 )\to \G{g}_4\oplus\G{g}_1,\quad \ad_{(w,\alpha)}{(w', \alpha')}= & ([\alpha,w'] + [w,\alpha'],[\alpha,\alpha']) \\
 =& (\ad_\a w' +\ad_w\a' , \ad_\a \a').
\end{align*}
Consequently, given any $(v,\b) \in \G{g}_2\oplus \G{g}_3$ and any $(w,\a) \in \G{g}_4\oplus \G{g}_1$, we have the mappings 
\begin{align}
&\mathfrak{b}_{(v, \beta)}=\G{g}_4\oplus\G{g}_1\to \G{g}_2\oplus\G{g}_3, \qquad(w,\alpha)\mapsto(w,\alpha) \rt (v,\beta) = (\ad_\alpha v,\ad_\alpha\beta),\label{b-action-Tok}\\
&\mathfrak{a}_{(w,\eta)}= \G{g}_2\oplus\G{g}_3\to \G{g}_4\oplus\G{g}_1, \qquad( v,\beta))\mapsto \psi((w,\alpha) ,( v,\beta))=(0,0),\label{a-action-Tok}\\
&{}_{(v,\beta)}\psi: \G{g}_4\oplus\G{g}_1\to\G{g}_4\oplus\G{g}_1, \qquad (w,\alpha)\mapsto {}_{(v,\beta)}\psi(w,\alpha)=(0,0),\\
& \t_{(v,\beta)}: \G{g}_2\oplus\G{g}_3 \to \G{g}_4\oplus\G{g}_1, \qquad \vec{w} \mapsto \t_{(v,\beta)}(v',\beta')=(-B_i(\ad_\beta v'+\ad_v\beta'),0),\\
 &\G{ad}_{(v,\beta)}:\G{g}_2\oplus\G{g}_3 \to \G{g}_2\oplus\G{g}_3, \qquad (v',\beta')\mapsto [[(v,\beta),(v',\beta')]]=(0,0),\\
 &\ad_{(w,\alpha)}:\G{g}_4\oplus\G{g}_1 \to\G{g}_4\oplus\G{g}_1, \quad(w',\alpha')\mapsto \ad_{(w,\alpha)}(w',\alpha')=(\ad_\a w' +\ad_w\a' , \ad_\a \a').\label{dual_ad_Tok}
\end{align}
Next, dualizing \eqref{left-action-Tok}-\eqref{right-action-Tok} we obtain
\begin{align}\label{Dual-left-action-Tok}
&\overset{\ast}{\vartriangleleft} : (\G{g}^{\ast}_2\oplus\G{g}_3^{\ast})\oplus (\G{g}_4\oplus\G{g}_1)\to \G{g}_2^{\ast}\oplus\G{g}_3^{\ast}, \qquad (\overline{v},\overline{\beta})\overset{\ast}{\vartriangleleft}(w,\alpha)  = (-\ad^*_{\alpha}{\overline{v}}, -\ad^*_{\alpha}{\overline{\beta}}),\\
&\phi_{(v,\beta)}^{\ast} :\G{g}_2^\ast\oplus\G{g}_3^\ast \to \G{g}_2^\ast\oplus\G{g}_3^\ast, \qquad \phi_{(v,\beta)}^\ast(\overline{v},\overline{\beta}) = (0,0), \label{Dual-phi-map-Tok} \\
&\t_{(v,\beta)}^{\ast} :\G{g}_4^\ast\oplus\G{g}_1^{\ast}\to \G{g}_2^\ast\oplus\G{g}_3^\ast , \qquad \t_{(v,\beta)}^{\ast}(\overline{w},\overline{\alpha})=(B_i\ad^\ast_{\beta} \overline{w}, B_i\ad^*_v\overline{w}),  \label{Dual-theta-map-Tok} \\
&{}_{(w,\alpha)}\psi^{\ast}: \G{g}^\ast_4\oplus\G{g}^\ast_1   \to \G{g}_2^\ast\oplus\G{g}_3^\ast, \qquad {}_{(w,\alpha)}\psi^\ast( \overline{w},\overline{\alpha})=(0,0), \label{Dual-right-action-Tok}
\end{align}
while the dualization of \eqref{b-action-Tok} and \eqref{a-action-Tok} yield
\begin{align}
&\mathfrak{b}^{\ast}_{(v, \beta)}=  \G{g}_2^{\ast}\oplus\G{g}_3^{\ast} \to \G{g}_4^{\ast}\oplus\G{g}_1^{\ast}, \qquad (\overline{v},\overline{\beta})\mapsto\mathfrak{b}^{\ast}_{(v,\beta)}(\overline{v},\overline{\beta}) = (0, \ad_v^{\ast}{\overline{v}}+\ad_{\beta}^{\ast}{\overline{\beta}}), \\
&\mathfrak{a}^{\ast}_{(w,\eta)}= \G{g}_4^{\ast}\oplus\G{g}_1^{\ast}\to \G{g}_2^{\ast}\oplus\G{g}_3^{\ast} , \qquad (\overline{w},\overline{\alpha})\mapsto\mathfrak{a}^{\ast}_{(w, \eta)}(\overline{w},\overline{\alpha})= (0,0).
\end{align} 
Finally the coadjoint action of the Lie algebra $\G{g}_4\oplus\G{g}_1$ on its dual $\G{g}_4^{\ast}\oplus\G{g}_1^{\ast}$ may be expressed by
\begin{equation}
\ad^{\ast}: (\G{g}_4\oplus\G{g}_1)\oplus (\G{g}_4^{\ast}\oplus\G{g}_1^{\ast} )\to \G{g}_4^{\ast}\oplus\G{g}_1^{\ast},\qquad \ad^{\ast}_{(w,\alpha)}(\overline{w},\overline{\alpha})=(-\ad^{\ast}_{\alpha}\overline{w}, -\ad^{\ast}_{w}\overline{w}-\ad^{\ast}_{\alpha}\overline{\alpha}).
\end{equation}
Substituting all these into \eqref{EPE} and \eqref{EPE-II}, we arrive at the Euler-Poincar\'{e} equations 
 \begin{align}
& \frac{d}{dt}\frac{\delta\ell}{\delta v}   = -\ad^{\ast}_{\alpha}\frac{\delta\ell}{\delta v} +B_i \ad^{\ast}_{\beta}\frac{\delta\ell}{\delta w},\\
& \frac{d}{dt}\frac{\delta\ell}{\delta \beta}   = -\ad^{\ast}_{\alpha}\frac{\delta\ell}{\delta\beta} +B_i \ad^{\ast}_{v}\frac{\delta\ell}{\delta w},\\
&\frac{d}{dt}\frac{\delta\ell}{\delta w}  =\ad^{\ast}_{\alpha}\frac{\delta\ell}{\delta w},\\
&\frac{d}{dt}\frac{\delta\ell}{\delta\alpha}  =\ad^{\ast}_{w}\frac{\delta\ell}{\delta w}+\ad^{\ast}_{\alpha}\frac{\delta\ell}{\delta \alpha}-\ad^{\ast}_{v}\frac{\delta\ell}{\delta v}-\ad^{\ast}_{\beta}\frac{\delta\ell}{\delta \beta}.
\end{align}
associated to a reduced Lagrangian $\ell = \ell(v,\b,w,\a)$ on $\G{g}_2\oplus \G{g}_3 \oplus \G{g}_4\oplus \G{g}_1$.

On the other extreme, given a reduced Hamiltonian $\G{H} = \G{H}(\overline{v},\overline{\beta},\overline{w},\overline{\alpha})$ on $\G{g}_2^{\ast}\oplus\G{g}_3^{\ast}\oplus \G{g}_4^{\ast}\oplus\G{g}_1^{\ast}$, the Lie-Poisson equations \eqref{LiePoissonEqn-I} and \eqref{LiePoissonEqn-II} take the form
\begin{align}
&\frac{d \overline{v}}{dt}=\ad^{\ast}_{\frac{\delta H }{\delta \overline{\alpha}}}\overline{v}-B_i \ad^{\ast}_{\frac{\delta H }{\delta \overline{\beta}}}\overline{w},\\
&\frac{d \overline{\beta}}{dt}=\ad^{\ast}_{\frac{\delta H }{\delta \overline{\alpha}}}\overline{\beta}- B_i \ad^{\ast}_{\frac{\delta H }{\delta \overline{v}}}\overline{w},\\
&\frac{d \overline{w}}{dt}=-\ad^{\ast}_{\frac{\delta H }{\delta \overline{\alpha}}}\overline{w},\\
&\frac{d \overline{\alpha}}{dt}=-\ad^{\ast}_{\frac{\delta H }{\delta \overline{w}}}\overline{w}-\ad^{\ast}_{\frac{\delta H }{\delta \overline{\alpha}}}\overline{\alpha}+\ad^{\ast}_{\frac{\delta H }{\delta \overline{v}}}\overline{v}+\ad^{\ast}_{\frac{\delta H }{\delta \overline{\beta}}}\overline{\beta}.
\end{align}


\bibliographystyle{plain}
\bibliography{references}{}

\end{document}